\documentclass[a4paper, DIV=12, 11pt]{scrartcl}
\usepackage[latin1]{inputenc}
\usepackage{amsmath}
\usepackage{amsfonts}
\usepackage{amssymb}
\usepackage{amsthm}
\usepackage{graphicx}
\usepackage{thmtools}
\usepackage{bbm}
\usepackage{hyperref}
\usepackage[bottom, marginal]{footmisc}

\usepackage{tikz}
\usetikzlibrary{matrix}

\declaretheoremstyle[headfont=\normalfont]{normalhead}
\newtheorem{lemma}{Lemma}[section]
\newtheorem{theorem}[lemma]{Theorem}
\newtheorem{proposition}[lemma]{Proposition}
\newtheorem{corollary}[lemma]{Corollary}
\newtheorem{definition}[lemma]{Definition}
\newtheorem{remark}[lemma]{Remark}

\newtheorem{maintheorem}{Theorem}

\newcommand{\R}{\mathbb{R}}

\DeclareMathOperator{\sign}{sign}
\DeclareMathOperator{\Val}{Val}
\DeclareMathOperator{\VConv}{VConv}
\DeclareMathOperator{\Conv}{Conv}
\DeclareMathOperator{\vol}{vol}
\DeclareMathOperator{\vsupp}{v-supp}
\DeclareMathOperator{\supp}{supp}

\DeclareMathOperator{\epi}{epi}

\DeclareMathOperator{\D}{\bar{D}}

\DeclareMathOperator{\GL}{GL}

\DeclareMathOperator{\CNC}{N^*}
\DeclareMathOperator{\GW}{GW}

\setkomafont{title}{\normalfont\Large}

\author{Jonas Knoerr}
\title{Smooth valuations on convex functions}
\date{}

\newcommand{\Addresses}{{
		\bigskip
		\footnotesize
		
		Jonas Knoerr, \textsc{Institute of Discrete Mathematics and Geometry, Technical University Vienna, Wiedner Hauptstrasse 8-10, 1040 Wien, Austria}\par\nopagebreak
		\textit{E-mail address}: \texttt{jonas.knoerr@tuwien.ac.at}
		
		\medskip
}}
	
\makeatletter
\def\blfootnote{\xdef\@thefnmark{}\@footnotetext}
\makeatother

\makeindex
\begin{document}
\maketitle
\begin{abstract}
	We construct valuations on the space of finite-valued convex functions using integration of differential forms over the differential cycle associated to a convex function. We describe the kernel of this procedure and show that the intersection of this space of \emph{smooth} valuations with the space of all continuous dually epi-translation invariant valuations on convex functions is dense in the latter. As an application, we obtain a description of $1$-homogeneous, continuous, dually epi-translation invariant valuations that are invariant with respect to a compact subgroup operating transitively on the unit sphere.
\end{abstract}
\blfootnote{2020 \emph{Mathematics Subject Classification}. 52B45, 26B25, 53C65.\\
	\emph{Key words and phrases}. Convex function, valuation on functions, differential cycle.\\
	Partially supported by DFG grant BE 2484/5-2.}
\tableofcontents

\section{Introduction}
\subsection{General background}
	Let $V$ be a finite-dimensional real vector space of dimension $n$ and let $\mathcal{K}(V)$ denote the space of convex bodies, i.e. the set of all non-empty, compact, convex subsets in $V$, which is a complete, locally compact metric space with respect to the Hausdorff metric. A functional $\mu:\mathcal{K}(V)\rightarrow \R$ is called a valuation if
	\begin{align*}
		\mu(K\cup L)+\mu(K\cap L)=\mu(K)+\mu(L)
	\end{align*}
	whenever $K,L,K\cup L\in\mathcal{K}(V)$. Examples of valuations include the Euler characteristic, the intrinsic volumes as well as mixed volumes, and valuations thus play an prominent role in complex geometry. The theory of continuous translation invariant valuations is especially rich. Following the groundbreaking results of Alesker \cite{Alesker:IrreducibilityThm}, many algebraic structures were discovered on this space \cite{Alesker:Product_polynomial,Alesker:Fourier_transform,Bernig_Fu:Convolution}, which are closely related to Integral geometric formulas, see for example \cite{Bernig_Fu:Hermitian_integral_geometry,Bernig_Fu_Solanes:complex_space_forms,Wannerer:module_unitary_area_measures}. Many of these results rely on different descriptions of the dense subspace of \emph{smooth valuations}, which may for example be represented by integrating certain translation invariant differential forms over the conormal cycle associated to a convex body. This also lead to an extension of the theory of valuations to valuations on manifolds \cite{Alesker:Valuations_on_manifolds_1,Alesker:Valuations_on_manifolds_2,Alesker_Fu:Valuations_on_manifolds3,Bernig_Broecker:Valuations_manifolds_Rumin_cohomology}.\\
	
	These advances in geometric valuation theory sparked interest in a corresponding theory for valuations on functions. Let $X$ be a class of real valued functions. A functional $\mu:X\rightarrow  \R$ is called a valuation if
	\begin{align*}
		\mu(f\vee h)+\mu(f\wedge h)=\mu(f)+\mu(h)
	\end{align*}
	for all $f,h\in X$ such that the pointwise maximum $f\vee h$ and minimum $f\wedge h$ belong to $X$.\\
	Assuming that the functions in $X$ are defined on some set $A$, the \emph{epi-graph} of $f\in X$ is given by $\epi(f):=\{(a,t)\in A\times\R: f(a)\le t\}$. Then 
	\begin{align*}
		&\epi(f\vee h)=\epi(f)\cap \epi(h), &&\epi(f\wedge h)=\epi(f)\cup \epi(h)
	\end{align*}
	for all $f,h\in X$. In this sense, a valuation on functions is a valuation on epi-graphs.\\
	In recent years, valuations on a variety of well-known classes of functions have been studied and classified, including Sobolev-spaces \cite{Ludwig:Fisher_information_valuations,Ludwig:valuations_sobolev,Ma:valuations_sobolev}, $\mathrm{L}^p$-spaces  \cite{Ludwig:covariance_matrices_valuations,Ober:Minkowski_valuations_on_Lq_spaces,Tsang:valuations_Lp,Tsang:minkowski_val_Lp}, quasi-concave functions \cite{Bobkov_Colesanti:quermassintegrals_quasi-concave_functions,Colesanti_Lombardi:valuations_quasi-concave,Colesanti_Lombardi_Parapatits:translation_invariant_valuations_quasi-concave}, Orlicz-spaces \cite{Kone:valuations_orlicz}, functions of bounded variation \cite{Wang:semivaluations_bounded_variations} and convex functions \cite{Alesker:valuations_convex_functions_Monge-Ampere,Cavallina_Colesanti:monotone_valuations_convex_functions,Colesanti_Ludwig_Mussnig:Hessian_valuations,Colesanti_Ludwig_Mussnig:minkowski,Colesanti_Ludwig_Mussnig:Valuations_convex_functions,Colesanti_Ludwig_Mussnig:homogeneous_decomposition,Knoerr:support_of_dually_epi-translation_invariant_valuations,Mussnig:SLn_invariant_super_coercive,Mussnig:volume_polar_volume_euler}.\\
	We will be interested in valuations on the space $\Conv(V,\R):=\{f:V\rightarrow\R: f\text{ convex}\}$, which is a closed subspace of the space of continuous functions on $V$ with respect to the topology of uniform convergence on compact subsets and thus a metrizable topological space. More precisely, let $\VConv(V)$ denote the space of all continuous valuations $\mu:\Conv(V,\R)\rightarrow \R$ that are \emph{dually epi-translation invariant}, i.e. that satisfy
		\begin{align*}
			\mu(f+\lambda+c)=\mu(f)\quad \text{for all }f\in\Conv(V,\R), \lambda\in V^*, c\in\R.
		\end{align*}
	This invariance property is intimately tied to translation invariance: For a convex function $f:V\rightarrow\R$, the Legendre transform of $f$ is given by
	\begin{align*}
		f^*(y):=\sup\limits_{x\in V}\langle y,x\rangle -f (x)\quad \text{for }y\in V^*,
	\end{align*}
	which defines a convex but not necessarily finite-valued function on $V^*$. The invariance properties above correspond to the invariance of the valuation with respect to translations of the epi-graph $\epi(f^*)$ in $V^*\times\R$. As a consequence, the classical McMullen decomposition for continuous translation invariant valuations on convex bodies induces a homogeneous decomposition of $\VConv(V)$, as shown by Colesanti, Ludwig and Mussnig in \cite{Colesanti_Ludwig_Mussnig:homogeneous_decomposition} (see also \cite{Knoerr:support_of_dually_epi-translation_invariant_valuations}): If $\VConv_k(V)$ denote the space of $k$-homogeneous valuations, i.e. the space of all valuations $\mu\in\VConv(V)$ that satisfy $\mu(tf)=t^k\mu(f)$ for all $t\ge 0$, $f\in\Conv(V,\R)$, then
	\begin{align*}
		\VConv(V)=\bigoplus\limits_{k=0}^n\VConv_k(V).
	\end{align*}
	The first examples of such valuations were constructed by Alesker in  \cite{Alesker:valuations_convex_functions_Monge-Ampere} using Monge-Amp\'ere-type operators (although he used the same idea before in \cite{Alesker:valuationss_non-comm_det_and_pluripotential} to construct certain invariant valuations on convex bodies). He showed that for every $B\in C_c(\R^n)$, $A_1,\dots,A_{n-k}\in C_c(\R^n,\mathrm{Sym}^2(\R^n))$ there exists a unique continuous valuation $\mu:\Conv(\R^n,\R)\rightarrow\R$ such that
		\begin{align*}
		\mu(f)=\int_{\R^n}B(x)\det(D^2f(x)[k],A_1(x),\dots,A_{n-k}(x))dx 
		\end{align*}
		for all $f\in\Conv(\R^n,\R)\cap C^2(\R^n)$, where $\mathrm{Sym}^2(\R^n)$ denotes the space of symmetric $n\times n$ matrices, $\det:(\mathrm{Sym}^2(\R^n))^n\rightarrow\R$ is the mixed discriminant and the Hessian $D^2f$ of $f$ is taken with multiplicity $k$ in this expression.\\	
	Although he used the name $\VConv(V)$ for the larger space of valuations invariant with respect to the addition of linear functionals (but not constant functions), these examples are obviously dually epi-translation invariant. Thus the results of \cite{Alesker:valuations_convex_functions_Monge-Ampere} also apply to the smaller space considered in this article. \\
	
	Similar functionals were also considered in \cite{Colesanti_Ludwig_Mussnig:Hessian_valuations} to construct rotation invariant valuations on $\Conv(\R^n,\R)$ using certain measures on the tagent bundle of $\R^n$, called \emph{Hessian measures} (see \cite{Colesanti_Hug:Hessian_measures}). These ideas where thereafter used in \cite{Colesanti_Ludwig_Mussnig:Hadwiger_for_convex_functions} to give a complete characterization of all rotation invariant valuations in $\VConv(\R^n)$ in terms of so called singular Hessian valuations.
	
\subsection{Results of the present article}
	The examples presented in the previous section admit a very simple geometric interpretation, at least if the functions involved are sufficiently regular: For $f\in\Conv(V,\R)\cap C^2(V)$ we can consider the graph of the differential $df:V\rightarrow V^*$ as an $n$-dimensional $C^1$-submanifold of the cotangent bundle $\pi:T^*V\rightarrow V$. If $V$ is an oriented vector space, this submanifold carries a natural orientation, so we may integrate suitable differential forms over it. It is easy to see that the Hessian measures can be obtained using this construction, however the same applies to the valuations considered by Alesker.\\
	The goal of this article is a systematic study of this type of valuation, which, the author hopes to establish, play a similar role to smooth translation invariant valuations on convex bodies.\\ 
	
	Let us first discuss the construction in slightly more detail. For a non-smooth convex function, we will replace integration over the graph of $df$ by integration with respect to the \emph{differential cycle} $D(f)$ as defined by Fu \cite{Fu:Monge-Ampere_1} (see also \cite{Jerrard:some_rigidity_results_related_to_MA_functions} for an extension of this construction). This $n$-current is uniquely determined by certain properties one would expect from the graph of the differential of $f\in\Conv(V,\R)$ and depends in fact continuously on $f$ in a suitable sense, see Theorem \ref{theorem_continuity_D_on_convex_functions} below.\\
	Let $\Omega_{hc}^n(T^*V)\subset \Omega^n(T^*V)$ denote the space of differential $n$-forms with \emph{horizontally compact support}, i.e. $\tau\in \Omega^n_{hc}(T^*V)$ if and only if there exists a compact subset $K\subset V$ such that $\supp\tau\subset \pi^{-1}(K)$, where $\pi:T^*V\rightarrow V$ denotes the natural projection.	In Section \ref{section_properties_differential_cycle} we show that $f\mapsto D(f)[\tau]$ defines a continuous valuation on $\Conv(V,\R)$ and we will call such valuations \emph{smooth}.\\
	Note that such a valuation is always invariant with respect to the addition of constants but not necessarily with respect to the addition of linear functionals. Let $\VConv_k(V)^{sm}$ denote the subspace of $\VConv_k(V)$ consisting of smooth valuations. Let us equip $\VConv_k(V)$ with the topology of uniform convergence on compact subsets in $\Conv(V,\R)$ (see \cite{Knoerr:support_of_dually_epi-translation_invariant_valuations} Proposition 2.4 for a description of these subsets). The main result of this article is
	\begin{maintheorem}
		\label{maintheorem_density_smooth_valuations}
		$\VConv_k(V)^{sm}$ is sequentially dense in $\VConv_k(V)$.
	\end{maintheorem}
	The proof of this result is based on the connection of the space of dually epi-translation invariant valuations on convex functions to translation invariant valuations on convex bodies obtained in \cite{Knoerr:support_of_dually_epi-translation_invariant_valuations}, in particular concerning the topology on both spaces. In addition, we require a description of the differential forms inducing smooth dually epi-translation invariant valuations. The key tool is the following description of the kernel of this construction of smooth valuations, which involves a certain second order differential operator $\D$ on $T^*V$, called the \emph{symplectic Rumin operator}(see Section \ref{section_kernel_theorem} for the precise definition). 
	\begin{maintheorem}
		\label{maintheorem_kernel_theorem}
		$\tau\in\Omega_{hc}^{n}(T^*V)$ satisfies $D(f)[\tau]=0$ for all $f\in\Conv(V,\R)$ if and only if
		\begin{enumerate}
			\item $\D\tau=0$,
			\item $\int_V\tau=0$, where we consider the zero section $V\hookrightarrow T^*V$ as a submanifold.
		\end{enumerate}
	\end{maintheorem}
	We apply these results to Hadwiger-type theorems for dually epi-translation invariant valuations on convex functions:	
	If $G\subset \GL(V)$ is a subgroup, then we call a valuation $\mu\in\VConv(V)$ $G$-invariant if $\mu(f\circ g)=\mu(f)$ for all $f\in\Conv(V,\R)$ and all $g\in G$. Then a result similar to Theorem \ref{maintheorem_density_smooth_valuations} holds for valuations that are invariant under a compact subgroup, see Proposition \ref{proposition:density_invariant_valuations}. By considering the relevant invariant differential forms and using the density of smooth valuations, we obtain 
	\begin{maintheorem}
		\label{maintheorem_classification_transitive_group_1-hom_case}
		Let $G\subset \mathrm{O}(n)$ be a closed subgroup that operates transitively on the unit sphere. If $\mu\in\VConv_1(\R^n)$ is a continuous valuation that is in addition $G$-invariant, then $\mu$ is $\mathrm{O}(n)$-invariant. 
	\end{maintheorem}
	As all continuous $\mathrm{O}(n)$-invariant valuations in $\VConv(\R^n)$ were classified by Colesanti, Ludwig and Mussnig \cite{Colesanti_Ludwig_Mussnig:Hadwiger_for_convex_functions}, this gives a complete characterization of these valuations in terms of singular Hessian valuations. 
	
	\subsection{Plan of the article}
		In Section \ref{section_translation_invariant_valuations} we recall some facts about translation invariant valuations on convex bodies, in particular the notion of vertical support introduced in \cite{Knoerr:support_of_dually_epi-translation_invariant_valuations}. The main results are Proposition \ref{proposition_approximation_by_smooth_val_in_Val_with_restrictions_on_vsupp}, which considers the problem of approximating continuous valuations by smooth valuations under restrictions on the vertical support, as well as Proposition \ref{proposition_characterization_support_GW_smooth_valuation}, which establishes a description of the vertical support for smooth valuations given by the integration of a differential form over the conormal cycle of a convex body.\\
		Section \ref{section_dually_epi-translation_invariant_valuations} recalls some facts on dually epi-translation invariant valuations obtained in \cite{Knoerr:support_of_dually_epi-translation_invariant_valuations}, in particular a connection between $\VConv(V)$ and the space of continuous translation invariant valuations on $\mathcal{K}(V^*\times\R)$.\\
		Smooth valuations on convex functions are introduced in Section \ref{section:Construction_valuations_differential_cycle}. Section \ref{section_properties_differential_cycle} recalls some properties of the differential cycle and we prove that our construction indeed defines continuous valuations on $\Conv(V,\R)$. Theorem \ref{maintheorem_kernel_theorem} is then proved in Section \ref{section_kernel_theorem}.\\
		We apply this theorem in Section \ref{section:invariant_smooth_valuations} to invariant smooth valuations. For this article, the most important results are Proposition \ref{proposition_desription_image_symp_D} and Theorem \ref{theorem:Characterization_smooth_dually_epi_valuations} from Section \ref{section:representation_smooth_dually_epi}, which give a characterization of smooth dually epi-translation invariant valuations in terms of differential forms.\\
		Finally, we prove Theorem \ref{maintheorem_density_smooth_valuations} in Section \ref{section_characterization_of_smooth_valuations} by establishing a relation between elements of $\VConv(V)^{sm}$ and smooth translation invariant valuations on $\mathcal{K}(V^*\times\R)$. Theorem \ref{maintheorem_classification_transitive_group_1-hom_case} is then proved in Section \ref{section:Application_to_invariant_valuations} using results from Section \ref{section:Invariance_under_subgroup}. 
		
		\paragraph{Acknowledgments}
		Part of this paper was written during a stay at the Universit\`a degli Studi di Firenze and I want to thank the university and especially Andrea Colesanti for the hospitality. I also want to thank Andreas Bernig for many useful discussions, suggestions and encouragement during this project, as well as his comments and remarks on the first draft of this paper. Finally I want to thank the referees for the careful checking of a first draft of this
		paper and many useful comments, especially for pointing out a serious error in one of the sections. This error could not be fixed with the tools used in this article and the relevant sections have thus been removed.\\

\section{Translation invariant valuations on convex bodies and their vertical support}
\subsection{Vertical support and approximation by smooth valuations}
	\label{section_translation_invariant_valuations}
		Let $\Val(V)$ denote the space of all continuous translation invariant valuations on $\mathcal{K}(V)$, i.e. all valuations $\mu:\mathcal{K}(V)\rightarrow \R$ that are continuous in the Hausdorff metric and satisfy $\mu(K+x)=\mu(K)$ for all $K\in\mathcal{K}(V)$, $x\in V$. This space becomes a Fr\'echet space with respect to the topology of uniform convergence on compact subsets. One of the most striking features is the following homogeneous decomposition.
			\begin{theorem}[\cite{McMullen:Euler_type_McMullen_decomposition}]
				For $k\in\R$ define $\Val_k(V):=\{\mu\in\Val(V): \mu(t K)=t^k\mu(K)\text{ for } K\in\mathcal{K}(V),t\ge 0\}$. Then
				\begin{align*}
					\Val(V)=\bigoplus\limits_{k=0}^{n}\Val_k(V).
				\end{align*}
			\end{theorem} 
		Elements of $\Val_k(V)$ are called \emph{$k$-homogeneous} or homogeneous of degree $k$. Note that the theorem implies that $\Val(V)$ is actually a Banach space with respect to the norm
		\begin{align*}
			\|\mu\|:=\sup\limits_{K\subset B}|\mu(K)|,
		\end{align*}
		where $B\in\mathcal{K}(V)$ is any convex body with non-empty interior.\\
		Let us restate the homogeneous decomposition: For any valuation $\mu\in\Val(V)$, the map $t\mapsto \mu(tK)$ is a polynomial in $t\ge 0$ with degree bounded by the dimension of $V$. Starting with a homogeneous element $\mu\in\Val_k(V)$, this theorem allows us to define the polarization of $\mu$ (see \cite{Schneider:convex_bodies_Brunn-Minkowski} Theorem 6.3.6). We obtain a symmetric functional $\bar{\mu}:\mathcal{K}(V)^k\rightarrow \R$ with the following properties:
		\begin{enumerate}
			\item $\bar{\mu}$ is a continuous valuation in each coordinate.
			\item $\bar{\mu}$ is additive in each coordinate: For $K,L, K_2,\dots, K_k\in\mathcal{K}(V)$:
			\begin{align*}
				\bar{\mu}(K+L,K_2,\dots,K_k)=\bar{\mu}(K,K_2,\dots,K_k)+\bar{\mu}(L,K_2,\dots,K_k).
			\end{align*}
			\item $\bar{\mu}(K,\dots,K)=\mu(K)$ for all $K\in\mathcal{K}(V)$.
		\end{enumerate}
		Note that the first two properties imply that $\bar{\mu}$ is $1$-homogeneous in each argument. In particular, $\bar{\mu}$ is essentially a multilinear functional. To make this precise, recall that any convex body $K\in\mathcal{K}(V)$ is uniquely determined by its support function $h_K:V^*\rightarrow\R$ given by
		\begin{align*}
			h_K(y):=\sup_{x\in K}\langle y,x\rangle\quad\text{for }y\in V^*.
		\end{align*}
		If the boundary of $K$ is smooth with strictly positive Gauss curvature, then $h_K\in C^\infty(V^*\setminus\{0\})$. As $h_K$ is $1$-homogeneous, it is advantageous to consider it as a function on the unit sphere $S(V^*)$ in $V^*$, assuming that we have chosen a scalar product.\\ In invariant terms this may be expressed in the following way: Let $\mathbb{P}_+(V^*)$ denote the space of oriented lines in $V^*$ and consider the line bundle $L$ over $\mathbb{P}_+(V^*)$ with fiber over $l\in\mathbb{P}_+(V^*)$ given by $L_l:=\{h:l^+\rightarrow\R: h\  1\text{-homogeneous}\}$. Here $l^+\subset l\setminus\{0\}$ is the positive half line induced by the orientation of $l$. Then any support function of a convex body $K\in\mathcal{K}(V)$ induces a continuous section of $L$, that we denote by $h_K$ again, by setting
		\begin{align*}
			[h_K(l)](y)=\sup\limits_{x\in K}\langle y,x\rangle\quad \text{for }y\in l^+.
		\end{align*}
		Let $C^\infty(\mathbb{P}_+(V^*),L)$ denote the space of all smooth sections of $L$. The polarization $\bar{\mu}$ can than be considered as a continuous multilinear functional $C^\infty(\mathbb{P}_+(V^*),L)^k\rightarrow \R$. Applying the Schwartz kernel theorem, we obtain the following result due to Goodey and Weil:
		\begin{theorem}[\cite{Goodey_Weil:Distributions_and_valuations}]
			\label{theorem_GW_for_Val}
			For every $\mu\in\Val_k(V)$ there exists a unique distribution $\GW(\mu)\in\mathcal{D}'(\mathbb{P}_+(V^*)^k,L^{\boxtimes k})$, called the Goodey-Weil distribution of $\mu$, such that
			\begin{align*}
				\GW(\mu)\left[h_{K_1}\otimes\dots\otimes h_{K_k}\right]=\bar{\mu}(K_1,\dots,K_k)
			\end{align*}
			for all smooth convex bodies $K\in\mathcal{K}(V)$ with strictly positive Gauss curvature. \\
			In particular, the map $\GW:\Val_k(V)\rightarrow\mathcal{D}'(\mathbb{P}_+(V^*)^k,L^{\boxtimes k})$ injective.
		\end{theorem}
		Here $\mathcal{D}'(\mathbb{P}_+(V^*)^k,L^{\boxtimes k})$ denotes the topological dual space of $C^\infty(\mathbb{P}_+(V^*)^k,L^{\boxtimes k})$, where $L^{\boxtimes k}$ is the line bundle over $\mathbb{P}_+(V^*)^k$ whose fibers consist of the tensor products of the corresponding fibers in $L$.\\
		In \cite{Knoerr:support_of_dually_epi-translation_invariant_valuations} the notion of vertical support was introduced for valuations in $\Val(V)$. It is based on the following observation by Alesker.
		\begin{proposition}[\cite{Alesker:McMullenconjecture} Proposition 3.3]
			For $\mu\in\Val_k(V)$, the support of $\GW(\mu)$ is contained in the diagonal in $\mathbb{P}_+(V^*)^k$.
		\end{proposition}
		\begin{definition}
			For $1\le k\le n$, we define the vertical support of $\mu\in\Val_k(V)$ to be the set 
			\begin{align*}
				\vsupp \mu:=\Delta^{-1}(\supp\GW(\mu)),
			\end{align*}
			where $\Delta:\mathbb{P}_+(V^*)\rightarrow \mathbb{P}_+(V^*)^k$ is the diagonal embedding.
			For $k=0$, we set $\vsupp\mu=\emptyset$. If $\mu=\sum_{i=0}^{n}\mu_i$ is the homogeneous decomposition, we set $\vsupp\mu:=\bigcup_{i=0}^n\vsupp\mu_i$.
		\end{definition}
		Alternatively one may use the following characterization of the vertical support.
		\begin{lemma}[\cite{Knoerr:support_of_dually_epi-translation_invariant_valuations} Lemma 6.14]
			\label{lemma_property_vertical_support}
			Let $\mu\in\Val(V)$. The vertical support is minimal (with respect to inclusion) among all closed sets $A\subset\mathbb{P}_+(V^*)$ with the following property: If $K,L\in\mathcal{K}(V)$ are two convex bodies with $h_K=h_L$ on a neighborhood of $A$, then $\mu(K)=\mu(L)$.
		\end{lemma}
		\begin{remark}
			While an invariant definition of the vertical support is necessary for the proof of Proposition \ref{proposition_approximation_by_smooth_val_in_Val_with_restrictions_on_vsupp} below, this invariant formulation becomes slightly cumbersome in later sections. If $V$ (and thus $V^*$) carries a Euclidean structure, we will consider the support function $h_K$ of $K\in\mathcal{K}(V)$ as a function on the unit sphere $S(V^*)\subset V^*$. Similarly, the vertical support is identified with a subset of the unit sphere in $V^*$ under the diffeomorphism $S(V^*)\cong \mathbb{P}_+(V^*)$ induced by the scalar product. Obviously, Lemma \ref{lemma_property_vertical_support} still holds under these identifications. Similarly, we will consider the Goodey-Weil distributions as distributions on $S(V^*)^k$, i.e. as elements of the topological dual space of $C^\infty(S(V^*)^k)$.
		\end{remark}
		For $A\subset \mathbb{P}_+(V^*)$ let $\Val_{k,A}(V)\subset \Val_k(V)$ denote the subspace of all valuations with vertical support contained in $A$. Lemma \ref{lemma_property_vertical_support} then directly implies
		\begin{corollary}[\cite{Knoerr:support_of_dually_epi-translation_invariant_valuations} Corollary 6.15]
			\label{corollary_Val_subspaces_compact_support_are_Banach}
			Let $A\subset \mathbb{P}_+(V^*)$ be closed. Then $\Val_{k,A}(V)\subset \Val_k(V)$ is closed and thus a Banach space.
		\end{corollary}
		
			We are mostly interested in the subspace of \emph{smooth} valuations. Recall that $\mu\in\Val(V)$ is called smooth if the map
		\begin{align*}
		\GL(V)&\rightarrow\Val(V)\\
		g&\mapsto \pi(g)\mu
		\end{align*} 
		is a smooth map, where $[\pi(g)\mu](K):=\mu(g^{-1}K)$ for $\mu\in\Val(V)$, $g\in\GL(V)$, $K\in\mathcal{K}(V)$ denotes the natural operation of $\GL(V)$ on $\Val(V)$. It is a standard fact from representation theory (using a convolution argument similar to the proof of Proposition \ref{proposition_approximation_by_smooth_val_in_Val_with_restrictions_on_vsupp} below) that the space of smooth valuations is dense in $\Val(V)$. It will be denoted by $\Val(V)^{sm}$ and is naturally equipped with a Fr\'echet topology that is stronger than the subspace topology.\\
		Recall that $\GL(V)$ operates on the line bundle $L\rightarrow\mathbb{P}_+(V^*)$ by $g(l,y):=(gl, y\circ g^{-1})$, where $gl$ is the oriented line induced by $\lambda\circ g^{-1}$ for $\lambda\in l^+$. This induces an operation of $\GL(V)$ on $C^\infty(\mathbb{P}_+(V^*),L)$ by setting $[(g f)(l)](y):=[f(g^{-1}l)](y\circ g)$ for $f\in C^\infty(\mathbb{P}_+(V^*),L)$, $l\in\mathbb{P}_+(V^*)$, $y\in l^+$. It is then easy to see that $\GW:\Val_k(V)\rightarrow\mathcal{D}'(\mathbb{P}_+(V^*)^k,L^{\boxtimes k})$ is $\GL(V)$-equivariant. In particular we obtain
		\begin{lemma}
			\label{lemma_vertical_support_under_GL}
			For $g\in \GL(V)$, $\mu\in\Val(V)$, $\vsupp(\pi(g)\mu)=g(\vsupp\mu)$.
		\end{lemma}
		\begin{proposition}
			\label{proposition_approximation_by_smooth_val_in_Val_with_restrictions_on_vsupp}
			Let $A\subset \mathbb{P}_+(V^*)$ be compact, $B\subset \mathbb{P}_+(V^*)$ a compact neighborhood of $A$. Then the following holds: For every $\mu\in\Val_{k,A}(V)$, there exists a sequence in $\Val_{k,B}(V)\cap \Val(V)^{sm}$ converging to $\mu$.
		\end{proposition}
		\begin{proof}
			Take a sequence of relatively compact open neighborhoods $(U_j)_j$ of the identity in $\GL(V)$
			such that their diameter converges to zero with respect to some Riemannian metric on $\GL(V)$. Given $\mu\in\Val_{k,A}(V)$ and $g\in U_j$, Lemma \ref{lemma_vertical_support_under_GL} shows that $\vsupp(\pi(g)\mu)\subset U_j\cdot A$. As $B$ is a neighborhood of $A$, the fact that the diameter of the neighborhoods $(U_j)_j$ converges to zero implies that there exists $N\in\mathbb{N}$ such that $U_j\cdot A\subset B$ for all $j\ge N$. In particular, $\pi(g)\mu\in \Val_{k,B}(V)$ for $g\in U_j$ and $j\ge N$ by Lemma \ref{lemma_vertical_support_under_GL}.\\
			Now take $\phi_j\in C_c^\infty(U_j)$ with $\int_{\GL(V)}\phi_j(g)dg=1$ (where have have equipped $\GL(V)$ with some left invariant Haar measure) and consider the valuations
			\begin{align*}
				\mu_j:=\int_{\GL(V)}\phi_j(g) \cdot\pi(g)\mu\ dg.
			\end{align*}
			For $j\ge N$, $\phi_j(g) \cdot\pi(g)\mu\in\Val_{k,B}(V)$ for all $g\in \GL(V)$ by construction. As this is a closed subspace by Corollary \ref{corollary_Val_subspaces_compact_support_are_Banach}, we deduce $\mu_j\in\Val_{k,B}(V)$ for all $j\ge N$.\\
			For $h\in \GL(V)$ and $j\ge N$, 
			\begin{align*}
				\pi(h)\mu_j=\int_{\GL(V)}\phi_j(g) \cdot\pi(hg)\mu\ dg=\int_{\GL(V)}\phi_j(h^{-1}g) \cdot\pi(g)\mu\ dg
			\end{align*}
			is just the convolution of the $\Val(V)$-valued continuous function $g\mapsto \pi(g)\mu$ on $\GL(V)$ and the smooth function $\phi_j\in C^\infty_c(\GL(V))$. In particular, $h\mapsto \pi(h)\mu_j$ depends smoothly on $h$ and $\mu_j$ is thus a smooth valuation, i.e. $\mu_j\in \Val_{k,B}(V)\cap \Val(V)^{sm}$ for $j\ge N$. Obviously, $(\mu_j)_j$ converges to $\mu$ in $\Val_k(V)$. The claim follows.
		\end{proof}
	
	\subsection{Construction of smooth valuations using the conormal cycle}
		\label{section_construction_smooth_valuations_co-normal_cycle}
		Let $V$ be an oriented Euclidean vector space. For a convex body $K\in\mathcal{K}(V)$, the set 
		\begin{align*}
			\CNC (K):=\left\{(x,v)\in V\times S(V^*): v\text{ outer unit normal to } K \text{ in }x\in\partial K\right\}
		\end{align*} is a Lipschitz submanifold of the co-sphere bundle $V\times S(V^*)$ of dimension $n-1$, which carries a natural orientation induced by the orientation of $V$ and can thus be considered as an integral current, called the \emph{conormal cycle} of $K$. For a precise definition of this current, we refer to \cite{Alesker_Fu:Valuations_on_manifolds3} (see also \cite{Fu:Curvature_measures_of_subanalytic_sets} for results on the conormal cycle for arbitrary compact as well as subanalytic sets). \\
		For special cases of convex bodies, the conormal cycle admits a very simple description. Recall that the support functional $h_K:V^*\setminus\{0\}\rightarrow\R$ is smooth for every smooth convex body $K\in\mathcal{K}(V)$ with strictly positive Gauss curvature. As it is $1$-homogeneous, its differential $dh_K:V^*\setminus\{0\}\rightarrow (V^*)^*\cong V$ is $0$-homogeneous and can thus be considered as a map $d'h_K:S(V^*)\rightarrow V$.
		\begin{lemma}
			\label{lemma_conormal-cycle-smooth-body}
			If $K\in\mathcal{K}(V)$ is smooth with strictly positive Gauss curvature, then
			\begin{align*}
				\CNC(K)=\left(d' h_K\times Id\right)_*\left[S(V^*)\right].
			\end{align*}
		\end{lemma}
		Let $I_{n-1}(V\times S(V^*))$ denote the space of integral currents of dimension $n-1$ in $V\times S(V^*)$. 
		\begin{proposition}[\cite{Alesker_Fu:Valuations_on_manifolds3} Proposition 2.1.12]
			\label{proposition_continuity_conormal_cycle}
			The map $\CNC:\mathcal{K}(V)\rightarrow I_{n-1}(V\times S(V^*))$ is continuous, where $I_{n-1}(V\times S(V^*))$ is equipped with the local flat metric topology. Furthermore, $\CNC$ is a valuation: For $K,L\in\mathcal{K}(V)$:
			\begin{align*}
				\CNC(K)+\CNC(L)=\CNC(K\cup L)+\CNC(K\cap L) \quad\text{if }K\cup L \in \mathcal{K}(V).
			\end{align*}
		\end{proposition}
		Let $\Omega^{n-1}(V\times S(V^*))^{tr}$ denote the space of translation invariant differential forms on $V\times S(V^*)$. From the previous proposition, one easily deduces that for a translation invariant differential form $\omega\in\Omega^{n-1}(V\times S(V^*))$, the map
		\begin{align*}
			\mathcal{K}(V)&\rightarrow \R\\
			K&\mapsto \CNC(K)[\omega]
		\end{align*}
		defines a continuous translation invariant valuation on $\mathcal{K}(V)$.\\
		Set $\Omega^{k,n-k-1}(V\times S(V^*))^{tr}:=\Lambda^k V\otimes \Omega^{n-k-1}(S(V^*))$. Then the space of translation invariant differential $(n-1)$-forms on $V\times S(V^*)$ decomposes as 
		\begin{align*}
			\Omega^{n-1}(V\times S(V^*))^{tr}=\bigoplus\limits_{k=0}^{n-1}\Omega^{k,n-k-1}(V\times S(V^*))^{tr}.
		\end{align*}
		\begin{theorem}[\cite{Alesker:Valuations_on_manifolds_1} Theorem 5.2.1]
			\label{theorem_representation_smooth_val_conv_bodies_using_normal_cycle}
			The map
			\begin{align*}
				\Omega^{k,n-k-1}(V\times S(V^*))^{tr}&\rightarrow \Val_k(V)^{sm}\\
				\omega&\mapsto\left(K\mapsto \CNC(K)[\omega]\right)
			\end{align*}
			is surjective for $0\le k\le n-1$.
		\end{theorem}
		The kernel of this map was described by Bernig and Bröcker in \cite{Bernig_Broecker:Valuations_manifolds_Rumin_cohomology}. It uses a certain second order differential operator $D:\Omega^{n-1}(V\times S(V^*))\rightarrow \Omega^n(V\times S(V^*))$ defined on the contact manifold $V\times S(V^*)$, which is called the \emph{Rumin operator} (see \cite{Rumin:formes_differential_contact}).
		\begin{theorem}[\cite{Bernig_Broecker:Valuations_manifolds_Rumin_cohomology} Theorem 2.2]
			\label{theorem_kernel_theorem_convex_bodies}
			$\omega\in \Omega^{k,n-k-1}(V\times S(V^*))^{tr}$ induces the trivial valuation if and only if
			\begin{enumerate}
				\item $D\omega=0$,
				\item $\pi_*\omega=0$.
			\end{enumerate}
			Here $\pi:V\times S(V^*)\rightarrow V$ is the natural projection and $\pi_*:\Omega^{n-1}(V\times S(V^*))^{tr}\rightarrow C^\infty(V)$ is the fiber integration.
		\end{theorem}
		Note that the first condition is always satisfied for $k=0$, while the second is always satisfied for $k\ne0$. \\
		
		As we will need it in the following sections, let us discuss the Rumin operator in more detail.\\
		On $V\times S(V^*)$ there exists a canonical distribution of hyperplanes $H\subset T(V\times S(V^*))$, given by 
		\begin{align*}
		H_{(x,v)}=\ker (v\circ d\pi|_{(x,v)}) \quad \text{for }(x,v)\in V\times S(V^*).
		\end{align*}
		In other words, this hyperplane distribution is given by the kernel of the no-where vanishing $1$-form $\alpha\in \Omega^1(V\times S(V^*))$:
		\begin{align*}
		\alpha|_{(x,v)}:=\langle v, d\pi\cdot \rangle\quad \text{for }(x,v)\in V\times S(V^*).
		\end{align*}
		The restriction of $d\alpha$ to each hyperplane is non-degenerate, so the distribution $H$ is a \emph{contact distribution}. Due to the non-degeneracy of $d\alpha$ on the contact hyperplanes, one can introduce a unique vector field $R$ on $V\times S(V^*)$ such that $i_R \alpha=1$, $i_Rd\alpha=0$, which is called the \emph{Reeb vector field}. Let us call a form $\omega\in\Omega^k(V\times S(V^*))$ \emph{vertical} if its restriction to the contact distribution vanishes. It is easy to see that this is equivalent to $\alpha\wedge\omega=0$. Plugging in $R$, this yields $\omega=\alpha\wedge i_R\omega$ for vertical differential forms.\\
		One can show that for any $\omega\in\Omega^{n-1}(V\times S(V^*))$, there exists a unique vertical form $\xi$ such that $d(\omega+\xi)$ is vertical. In this case, the Rumin operator of $\omega$ is defined as $D\omega:=d(\omega+\xi)$.\\
		\\
		
		The proof of Theorem \ref{theorem_kernel_theorem_convex_bodies} actually shows the following
		\begin{lemma}
			\label{lemma:Domega=0_relation_conormal_cycle}
			For an open subset $U\subset S(V^*)$ and $\omega\in \Omega^{n-1}(V\times S(V^*))$ the following are equivalent
			\begin{enumerate}
				\item $D\omega=0$ on $V\times U$.
				\item $\CNC(K)[\pi_2^*\phi\wedge i_RD\omega]=0$ for all $K\in\mathcal{K}(V)$, $\phi\in C^\infty(S(V^*))$ with $\supp \phi\subset U$. Here $\pi_2:V\times S(V^*)\rightarrow S(V^*)$ denotes the projection onto the second factor.
			\end{enumerate}
		\end{lemma}
		This result relies on a version of the following observation, which was stated in this form by Wannerer in the proof of \cite{Wannerer:module_unitary_area_measures} Proposition 2.2.:
		\begin{proposition}
			\label{proposition:formula_first_variation}
			If $\mu\in \Val_k(V)$ is a smooth valuation represented by $\omega\in\Omega^{k,n-1-k}(V\times S(V^*))$, then for all $K,L\in\mathcal{K}(V)$
			\begin{align*}
				\frac{d}{dt}\Big|_0\mu(K+tL)=\CNC(K)[\pi_2^*h_L\wedge i_RD\omega].
			\end{align*}
		\end{proposition}

		In the final part of this subsection, we are going to establish a relation between the vertical support of a smooth valuation and the differential from representing it. 
		\begin{proposition}
				\label{proposition_characterization_support_GW_smooth_valuation}
				For $1\le k\le n-1$ let $\omega\in\Omega^{k,n-k-1}(V\times S(V^*))^{tr}$ represent a smooth valuation $\mu\in\Val_k(V)^{sm}$. Then $\vsupp \mu = \pi_2(\supp D\omega)$.
			\end{proposition}
			\begin{proof}
				Let us start by showing 
				\begin{align}
					\label{equation:GW_conormal_cycle}
					\GW(\mu)(\phi_1\otimes\dots\otimes \phi_k)=&\frac{1}{k!}\CNC(K)\left[ (\pi_2^*\phi_k\wedge i_RD)\dots(\pi_2^*\phi_{1}\wedge i_RD)\omega\right]
				\end{align}
				for any $K\in\mathcal{K}(V)$, $\phi_1,\dots,\phi_k\in C^\infty(S(V^*))$.
				Iterating the formula in Proposition \ref{proposition:formula_first_variation} for smooth convex bodies $L_1,\dots, L_k$ with strictly positive Gauss curvature, we obtain
				\begin{align*}
					\GW(\mu)(h_{L_1}\otimes\dots\otimes h_{L_k})=&\frac{1}{k!}\frac{\partial}{\partial \lambda_1}\Big|_0\dots\frac{\partial}{\partial \lambda_k}\Big|_0\CNC\left(K+\sum_{i=1}^{k}\lambda_iL_i\right)[\omega]\\
					=&\frac{1}{k!}\CNC(K)\left[(\pi_2^*h_{L_k}\wedge i_RD)\dots(\pi_2^*h_{L_1}\wedge i_RD)\omega\right].
				\end{align*}
				As $\GW(\mu)$ is uniquely determined by its values on functions of the form $h_{L_1}\otimes\dots\otimes h_{L_k}$, we see that
				\begin{align*}
					\GW(\mu)(\phi_1\otimes\dots\otimes \phi_k)=\frac{1}{k!}\CNC(K)\left[(\pi_2^*\phi_k\wedge i_RD)\dots(\pi_2^*\phi_{1}\wedge i_RD)\omega\right]
				\end{align*}
				for any $K\in\mathcal{K}(V)$.\\

				Let us show that $\vsupp \mu\subset \pi_2(D\omega)$: Assume that one of the functions $\phi_i$ in the equation above satisfies $\supp \phi_i\cap\pi_2(D\omega)=\emptyset$. As the Goodey-Weil distribution of a valuation is symmetric, we can assume $i=k$, so $\pi_2^*\phi_k\wedge i_RD\omega=0$. Thus the formula above implies $\GW(\mu)(\phi_1\otimes\dots\otimes\phi_k)=0$, and we see that $\supp \GW(\mu)\subset \Delta (\pi_2(\supp D\omega))$, i.e. $\vsupp \mu\subset \pi_2(\supp D\omega)$.\\
				
				For the converse inclusion, fix a point $v\in \pi_2(\supp D\omega)$ and let $U$ be an arbitrary neighborhood of $v$.  We will construct functions $\phi_1,\dots,\phi_k$ with support in $U$ such that $\GW(\mu)(\phi_1\otimes\dots\otimes\phi_k)\ne0$.\\
				Let us start with $\phi_1$: Us $D\omega$ does not vanish identically on $V\times U$ by assumption, Lemma \ref{lemma:Domega=0_relation_conormal_cycle} implies that there exists $\phi\in C^\infty(S(V^*))$ with supp $\phi\subset U$ such that $\CNC(K)[\pi_2^*\phi_1\wedge i_RD\omega]\ne 0$ for some $K\in\mathcal{K}(V)$. In particular, the valuation induced by the differential form $\omega_1:=\pi_2^*\phi_1\wedge i_RD\omega$ is non-trivial. By construction, this differential form is of bidegree $(k-1,n-k)$, i.e. it defines a $(k-1)$-homogeneous valuation. Thus $D\omega_1\ne 0$ if $k\ne 1$ by Theorem \ref{theorem_kernel_theorem_convex_bodies}. Obviously, $\pi_2(\supp D\omega_1)\subset U$.\\
				Repeating this construction, we obtain functions $\phi_1,\dots,\phi_{k}$ with the properties
				\begin{enumerate}
					\item $\supp \phi_i\subset U$,
					\item $\omega_{i+1}:=\pi_2^*\phi_{i+1}\wedge i_RD\omega_i$ defines a non-trivial valuation of degree $k-i-1$.
				\end{enumerate}
				In particular, the map
				\begin{align*}
				K\mapsto \CNC(K)[\omega_k]=\CNC(K)\left[(\pi_2^*\phi_k\wedge i_RD)\dots(\pi_2^*\phi_1\wedge i_RD)\omega\right]
				\end{align*}
				defines a $0$-homogeneous, non-trivial valuation, i.e. it is a constant multiple of the Euler characteristic. Using the expression of the Goodey-Weil distribution in \eqref{equation:GW_conormal_cycle}, we obtain
				\begin{align*}
				\GW(\mu)(\phi_1\otimes\dots\otimes \phi_k)=\frac{1}{k!}\CNC(K) \left[(\pi_2^*\phi_k\wedge i_RD)\dots(\pi_2^*\phi_{1}\wedge i_RD)\omega\right]\ne 0
				\end{align*}
				for any $K\in \mathcal{K}(V)$. As this is true for any neighborhood $U$ of $v$, $\Delta(v)\in \supp \GW(\mu)$, i.e. $v\in \vsupp(\mu)$.
			\end{proof}
			
\section{Dually epi-translation invariant valuations: Support and relation to valuations on convex bodies}
	\label{section_dually_epi-translation_invariant_valuations}
	The homogeneous decomposition for $\VConv(V)$ can be used to construct a version of the Goodey-Weil distributions for homogeneous dually epi-translation invariant valuations. Similar to the definition of the vertical support for translation invariant valuations, this may be used to define the support $\supp\mu\subset V$ of $\mu\in\VConv(V)$. We refer to \cite{Knoerr:support_of_dually_epi-translation_invariant_valuations} for details and only state the following characterization of the support, which may also be used as a definition.
		 \begin{proposition}[\cite{Knoerr:support_of_dually_epi-translation_invariant_valuations} Proposition 6.3]
		 	\label{proposition_support_convex_valuation}
		 	The support of $\mu\in\VConv(V)$ is minimal (with respect to inclusion) among the closed sets $A\subset V$ with the following property: If $f,g\in\Conv(V,\R)$ satisfy $f=g$ on an open neighborhood of $A$, then $\mu(f)=\mu (g)$.
		 \end{proposition}
	 	Moreover, the support of such a valuation is always a compact subset, see \cite{Knoerr:support_of_dually_epi-translation_invariant_valuations} Corollary 6.2.
		It turns out that the support is a useful concept that simplifies the topology of $\VConv(V)$. For $A\subset V$ compact, let $\VConv_{k,A}(V)$ denote the space of all valuations $\mu\in\VConv(V)$ with $\supp\mu\subset A$. Then we have the following result:
		\begin{proposition}[\cite{Knoerr:support_of_dually_epi-translation_invariant_valuations} Corollary 6.10]
			\label{proposition_Banach_structures_subspaces_VConv}
			For a compact subset $A\subset V$, the relative topology of $\VConv_{k,A}(V)\subset \VConv(V)$ is induced by a continuous norm. Moreover, this norm equips $\VConv_{k,A}(V)$ with the structure of a Banach space.
		\end{proposition}
		For a definition of this norm, see \cite{Knoerr:support_of_dually_epi-translation_invariant_valuations} Proposition 6.8.\\
		
		Our characterization of smooth valuations on convex functions is based on the following relation between valuations on convex functions and valuations on convex bodies.	.
		\begin{theorem}[\cite{Knoerr:support_of_dually_epi-translation_invariant_valuations} Theorem 3.4]
			For $\mu\in \VConv_k(V)$ consider $T(\mu)\in\Val_k(V^*\times\R)$ defined by
			\begin{align*}
				T(\mu)[K]:=\mu\left(h_K(\cdot,-1)\right)\quad\text{for } K\in\mathcal{K}(V^*\times\R).
			\end{align*}
			Then $T:\VConv_k(V)\rightarrow\Val_k(V^*\times\R)$ is well defined, continuous and injective.
		\end{theorem}
		The image of this map was described in \cite{Knoerr:support_of_dually_epi-translation_invariant_valuations} in terms of the supports of these valuations. To simplify the notation, assume that $V$ carries a Euclidean structure, which induces a Euclidean structure on $V\times\R$. Consider the smooth map
		\begin{align*}
		P:V&\rightarrow S(V\times\R)\\
		x&\mapsto \frac{1}{\sqrt{1+|x|^2}}(x,-1),
		\end{align*}
		which is a diffeomorphism onto its image.
		\begin{theorem}[\cite{Knoerr:support_of_dually_epi-translation_invariant_valuations} Theorem 6.17]
			\label{theorem_image_VCONV->Val}
			The image of $T:\VConv_k(V)\rightarrow\Val_k(V^*\times\R)$ consists precisely of the valuations in $\Val_k(V^*\times\R)$ whose vertical support is contained in the negative half sphere $S(V\times\R)_-:=\{(x,t)\in S(V\times\R):t<0\}$. Moreover $T:\VConv_{k,A}(V)\rightarrow\Val_{k,P(A)}(V^*\times\R)$ is a topological isomorphism between Banach spaces for any compact subset $A\subset V$.
		\end{theorem}
		
\section{Construction of valuations using the differential cycle}
	\label{section:Construction_valuations_differential_cycle}
	\subsection{Properties of the differential cycle}
	\label{section_properties_differential_cycle}
		In this section we summarize the basic facts concerning Monge-Ampère functions established by Fu in \cite{Fu:Monge-Ampere_1}. Let $V$ be an oriented vector space with volume form $\vol\in \Lambda^n V^*$ and let $\omega_s$ denote the natural symplectic form on $T^*V$. For simplicity, let us equip $V$ with a scalar product inducing the volume form $\vol$.
		\begin{theorem}[\cite{Fu:Monge-Ampere_1} Theorem 2.0]
			\label{theorem_characterization_differential_cycle}
			Let $f:V\rightarrow\mathbb{R}$ be a locally Lipschitzian function. There exists at most one integral current $S\in I_n(T^*V)$ such that
			\begin{enumerate}
				\item $S$ is closed, i.e. $\partial S=0$,
				\item $S$ is Lagrangian, i.e. $S\llcorner \omega_s=0$,
				\item $S$ is locally vertically bounded, i.e. $\supp S\cap \pi^{-1}(A)$ is compact for all $A\subset V$ compact,
				\item $S(\phi(x,y)\pi^*\vol)=\int_{V}\phi(x,df(x))d\vol(x)$ for all $\phi\in C^\infty_c(T^*V)$.
			\end{enumerate}
			Note that the right hand side of the last equation is well defined due to Rademacher's theorem.
		\end{theorem}
		If such a current exists, the function $f$ is called Monge-Amp\`ere. The corresponding current is denoted by $D(f)$ (it is denoted by $[df]$ in \cite{Fu:Monge-Ampere_1}) and is called the \emph{differential cycle of $f$}. Moreover, we have the following description of the support of $D(f)$: Let $\partial^*f:V\rightarrow \mathcal{K}(V^*)$ denote the unique upper semi-continuous multifunction with values in $\mathcal{K}(V^*)$ such that $df(x)\in\partial^*f(x)$ whenever $f$ is differentiable at $x\in V$ (also called the generalized differential by Clarke \cite{Clarke:optimization_non_smooth_analysis}).
		\begin{theorem}[\cite{Fu:Monge-Ampere_1} Theorem 2.2.]
			\label{theorem_FU_support_Differential_cycle}
			If $f:V\rightarrow\R$ is Monge-Amp\`ere, then 
			\begin{align*}
				\supp D(f)\subset \text{graph }\partial^*f:=\left\{(x,y)\in T^*V: y\in\partial^*f(x)\right\}.
			\end{align*}
			In particular, given an open set $U\subset V$,
			\begin{align*}
				\supp D(f)\cap\pi^{-1}(U)\subset U\times B_{\mathrm{lip}(f|_U)}(0),
			\end{align*}
			where $\mathrm{lip}(f|_U)$ denotes the Lipschitz-constant of $f|_U$.
		\end{theorem}
		Let us summarize some additional properties.
		\begin{proposition}[\cite{Fu:Monge-Ampere_1} Proposition 2.4]
			\label{proposition_Fu_sum_MA_andC11}
			Let $f$ be a Monge-Amp\`ere function and $\phi\in C^{1,1}(V)$. Then $f+\phi$ is Monge-Amp\`ere and 
			\begin{align*}
			F(f+\phi)=G_{\phi*}D(f),
			\end{align*}
			where $G_\phi:T^*V\rightarrow T^*V$ is given by $(x,y)\mapsto (x,y+d\phi(x))$.
		\end{proposition}
		\begin{proposition}
			\label{proposition_Fu_differential_cycle_and_diffeomorphisms}
			Let $\phi:V\rightarrow V$ be a diffeomorphism of class $C^{1,1}$. Then $f\circ\phi$ is Monge-Amp\`ere and
			\begin{align*}
				D(f\circ\phi)=\left(\phi^\#\right)_*D(f),	
			\end{align*}
			if $\phi$ is orientation preserving and
			\begin{align*}
				D(f\circ\phi)=-\left(\phi^\#\right)_*D(f),
			\end{align*}
			if $\phi$ is orientation reversing. Here $\phi^\#:T^*V\rightarrow T^*V$ is given by $(x,y)\mapsto (\phi^{-1}(x),\phi^*y)$.
		\end{proposition} 
		For orientation preserving diffeomorphisms this was shown in \cite{Fu:Monge-Ampere_1} Proposition 2.5. The second case follows with the same argument, taking into account that the last property in Theorem \ref{theorem_characterization_differential_cycle} requires an additional sign.\\
		Also note that \cite{Fu:Monge-Ampere_1} Remark 2.1 shows that for any $c\in\R\setminus\{0\}$ and any Monge-Amp\`ere function $f$ the function $cf$ is Monge-Amp\`ere with
		\begin{align}
			\label{equation_multiples_MA_functions}
			D(cf)=C_{*}D(f),
		\end{align}
		where $C:T^*V\rightarrow T^*V$ is given by $(x,y)\mapsto (x,cy)$.
		
		The differential cycle satisfies the following valuation property:
		\begin{proposition}[\cite{Fu:Monge-Ampere_1} Proposition 2.9]
			\label{proposition_Fu_valuation_property_Differential_cycle}
				Let $f,g:V\rightarrow\R$ be locally Lipschitzian. If any three of $f$, $g$, $f\vee g$ and $f\wedge g$ are Monge-Amp\`ere, then so is the fourth, and
				\begin{align*}
					D(f)+D(g)=D(f\wedge g)+D(f\vee g).
				\end{align*}
		\end{proposition}

		By \cite{Fu:Monge-Ampere_1} Proposition 3.1 all convex functions are Monge-Amp\`ere. We will now show that $D:\Conv(V,\R)\rightarrow I_n(T^*V)$ is continuous with respect to the local flat metric topology on the space  $I_n(T^*V)$ of integral currents of dimension $n$ on $T^*V$, i.e. the topology induced by the family of semi-norms
		\begin{align*}
			\|T\|_{A,\flat}&:=\sup\left\{|T(\omega)| : \supp \omega\subset A, \left\|\omega\right\|^\flat\le 1\right\} &&\text{for }T\in I_n(T^*V)\text{, where}\\
			\left\|\omega\right\|^\flat&:=\max (\|\omega\|_0,\|d\omega\|_0) &&\text{for }\omega\in\Omega^n(T^*V),
		\end{align*} 
		for $A\subset T^*V$ compact, where $\|\cdot\|_0$ denotes the $C^0$-norm. Recall also that the mass $M_U(T)$ of an $n$-current $T$ on an open subset $U$ is defined by
		\begin{align*}
			M_U(T):=\sup\{|T(\omega)|: \supp\omega\subset U, \|\omega\|_0\le 1\}.
		\end{align*}
		The proof is based on the following approximation result.
		\begin{proposition}[\cite{Fu:Monge-Ampere_1} Proposition 2.7.]
			\label{proposition_Fu_convergence_differential_cycle}
			Let $f_1,f_2,\dots:V\rightarrow\R$ be a sequence of Monge-Amp\`ere functions, and suppose that for each bounded open subset $U\subset V$ there exists a constant $C$ such that
			\begin{enumerate}
				\item $\mathrm{lip}(f_j|_U)\le C$
				\item $M_{\pi^{-1}(U)}(D(f_j))\le C$
			\end{enumerate}
			for all $j\in\mathbb{N}$. If $f=\lim\limits_{j\rightarrow\infty}f_j$ in the $C^0$-topology, then $f$ is Monge-Amp\`ere, with
			\begin{align*}
				D(f)=\lim\limits_{j\rightarrow\infty}D(f_j)
			\end{align*}
			in the local flat metric topology.
		\end{proposition}
		The two necessary bounds are established by the following two lemmas.
		\begin{lemma}
			\label{lemma_convex_functions_local_lipschitz_constants}
			Let $f:V\rightarrow\mathbb{R}$ be a convex function, and let $B_R$ denote the ball of radius $R>0$ in $V$. Then $f|_{B_R}$ is Lipschitz-continuous with Lipschitz-constant bounded by $2\sup_{|x|\le R+1}|f(x)|$.
		\end{lemma}
		\begin{proof}
			This is a special case of \cite{Rockafellar_Wets:Variational_analysis} 9.14.
		\end{proof}
		For $R>0$, let $U_R\subset V$ denote the open ball of radius $R$ centered at the origin.
		\begin{lemma}
			\label{lemma_mass_estimate_Differential_cycle}
			For $f\in\Conv(V,\R)$, $M_{\pi^{-1}(U_R)}(D(f))\le2^n\omega_n \left(\sup_{|x|\le R+1}|f(x)|\right)^n$.
		\end{lemma}
		\begin{proof}
			We will prove the following estimate:
			\begin{align*}
				M_{\pi^{-1}(U_R)}(D(f+\epsilon |\cdot|^2))\le2^n\omega_n\left(\sup_{|x|\le R+1}\left|f(x)+\epsilon |x|^2\right|\right)^n(1+\epsilon)^{\frac{n}{2}}.
			\end{align*}
			As $D(f+\epsilon |\cdot|^2)=G_{\epsilon*}D(f)$ for $G_\epsilon(x,y)=(x,y+2\epsilon x)$ due to Proposition \ref{proposition_Fu_sum_MA_andC11}, we see that $D(f+\epsilon |x|^2)$ converges to $D(f)$ weakly for $\epsilon\rightarrow0$, so the claim follows from this inequality using the lower semi-continuity of the mass norm.\\
			For some $\phi\in C^\infty_c(U_1,[0,\infty))$ with $\int_{V}\phi(x) d\vol(x)=1$ set $\phi_h(x):=h^{-n}\phi(\frac{x}{h})$ and consider the convolution $f_h:=(f+\epsilon |\cdot|^2)*\phi_h$ for $h>0$. Fu observed in the proof of \cite{Fu:Monge-Ampere_1} Proposition 3.1. that in this case 
			\begin{align*}
			M_{\pi^{-1}(U)}(D(f_h)\le \omega_n r^n(1+\epsilon)^{\frac{n}{2}}
			\end{align*}
			for any bounded open subset $U\subset V$, where $r>0$ can be chosen to be the Lipschitz-constant of $f+\epsilon |\cdot|^2$ on $\{x\in V:d(x,U)<h\}$. For $U=U_R$, we may thus choose $r=2\sup_{|x|\le R+1+h}\left|f(x)+\epsilon |x|^2\right|$ by Lemma \ref{lemma_convex_functions_local_lipschitz_constants}. The proof of \cite{Fu:Monge-Ampere_1} Proposition 3.1. shows that $D(f_h)\rightarrow D(f+\epsilon|\cdot|^2)$ in the local flat metric topology for $h\rightarrow0$, so in particular, $D(f_h)$ converges weakly to $D(f+\epsilon |\cdot|^2)$. The lower semi-continuity of the mass norm thus implies
			\begin{align*}
				M_{\pi^{-1}(U_R)}(D(f+\epsilon |\cdot |^2))\le \omega_n \left(2\sup_{|x|\le R+1}\left|f(x)+\epsilon |x|^2\right|\right)^n(1+\epsilon)^{\frac{n}{2}}.
			\end{align*}
		\end{proof}
		\begin{theorem}
			\label{theorem_continuity_D_on_convex_functions}
			$D:\Conv(V,\R)\rightarrow I^n(T^*V)$ is continuous with respect to the local flat metric topology on $I^n(T^*V)$.
		\end{theorem}
		\begin{proof}
			If $f_j\rightarrow f$ in $\Conv(V,\R)$, their Lipschitz-constants are locally uniformly bounded by Lemma \ref{lemma_convex_functions_local_lipschitz_constants}. The mass estimate from Lemma \ref{lemma_mass_estimate_Differential_cycle} shows that the mass of $D(f_j)$ is locally uniformly bounded as well. Thus $D(f_j)\rightarrow D(f)$ in the local flat metric topology by Proposition \ref{proposition_Fu_convergence_differential_cycle}.
		\end{proof}
		Let $\Omega^k_{hc}(T^*V)$ denote the space of all smooth $k$-forms $\tau$ on $T^*V$ with horizontally compact support, i.e. such that $\supp\tau\subset \pi^{-1}(K)$ for some compact set $K\subset V$.
		\begin{corollary}
			\label{corollary_Continuity_valuations_defined_by_differential_cycle}
			For each $\tau\in\Omega^n_{hc}(T^*V)$, $f\mapsto D(f)[\tau]$ defines a continuous valuation on $\Conv(V,\R)$.
		\end{corollary}
		\begin{proof}
			Let $K\subset V$ be a compact subset with $\supp\tau\subset \pi^{-1}(K)$. As the support of the differential cycle is vertically bounded, $\supp D(f)\cap \supp\tau$ is compact for every $f\in\Conv(V,\R)$, so $D(f)[\tau]$ is well defined for all $f\in\Conv(V,\R)$. Furthermore, Proposition \ref{proposition_Fu_valuation_property_Differential_cycle} shows that this functional satisfies the valuation property.\\
			To see that it is continuous, let $(f_j)_j$ be a sequence in $\Conv(V,\R)$ converging to $f\in\Conv(V,\R)$ uniformly on compact subsets. Choose $R>0$ such that $U_R$ contains $K$. As $f_j\rightarrow f$ uniformly on $U_R$, Lemma  \ref{lemma_convex_functions_local_lipschitz_constants} implies that the Lipschitz-constants of these functions are bounded on $U_R$ by some $L>0$. Now Theorem \ref{theorem_FU_support_Differential_cycle} shows that $\supp D(f_j)\cap\pi^{-1}(U_R)\subset U_R\times B_L(0)$. Let $A$ be a compact neighborhood of $B_R\times B_L(0)$ and $\phi\in C^\infty_c(T^*V)$ a function with $\phi=1$ on a neighborhood of $U_R\times B_L(0)$ and $\supp\phi\subset A$. Then
			\begin{align*}
				\left|D(f_j)[\tau]-D(f)[\tau]\right|=&\left|\left(D(f_j)-D(f)\right)[\phi\cdot \tau]\right|\le \left\|D(f_j)-D(f)\right\|_{A,\flat}\cdot \left\|\phi\cdot\tau\right\|^\flat.
			\end{align*}
			Now the claim follows from Theorem \ref{theorem_continuity_D_on_convex_functions}, as $||\phi\cdot \tau||^\flat<\infty$.
		\end{proof}
	\subsection{Kernel theorem}
		\label{section_kernel_theorem}
		By the previous section, any $\tau\in\Omega_{hc}^n(T^*V)$ defines a continuous valuation $\Conv(V,\R)\rightarrow\R$, $f\mapsto D(f)[\tau]$. To decide which differential forms induce the trivial valuation, we will need a symplectic version of the Rumin operator. The construction is based on the following decomposition (see \cite{Huybrechts:Complex_geometry} Proposition 1.2.30), also called \emph{Lefschetz decomposition}.
		\begin{proposition}
			\label{proposition_Lefschetz-decomposition}
			Let $(W,\omega_s)$ be a symplectic vector space of dimension $2n$ and let $L:\Lambda^*W^*\rightarrow \Lambda^*W^*$, $\tau\mapsto \omega_s\wedge\tau$ be the Lefschetz operator. For $0\le k\le n$ let $P^k:=\{\tau\in\Lambda^kW^*:L^{n-k+1}\tau=0\}$ denote the space of primitive forms. Then the following holds:
			\begin{enumerate}
				\item There exists a direct sum decomposition $\Lambda^kW^*=\bigoplus_{i\ge 0}L^{i}P^{k-2i}$.
				\item $L^{n-k}:\Lambda^kW^*\rightarrow\Lambda^{2n-k}W^*$ is an isomorphism.
			\end{enumerate}
		\end{proposition}
		Let $(M,\omega_s)$ be a symplectic manifold of dimension $2n$. Proposition \ref{proposition_Lefschetz-decomposition} implies that the Lefschetz operator $L:\Omega^{n-1}(M)\rightarrow\Omega^{n+1}(M)$ is an isomorphism.
		\begin{definition}
			For a symplectic manifold $(M,\omega_s)$ we define
			\begin{align*}
				&\bar{d}:\Omega^{n}(M)\rightarrow\Omega^{n-1}(M), &&\bar{d}\tau:=L^{-1}d\tau\\
				&\D:\Omega^n(M)\rightarrow\Omega^n(M), &&\D\tau:=d\bar{d}\tau=dL^{-1}d\tau
			\end{align*} 
			and call $\D$ the \emph{symplectic Rumin operator}.
		\end{definition}
		Note that $\bar{d}$ is a first order differential operator, while $\D$ is of second order.
		\begin{proposition}
			\label{proposition_properties_tilde_D_d}
			$\D$ and $\bar{d}$ have the following properties:
			\begin{enumerate}
				\item $\D\tau$ is primitive for all $\tau\in \Omega^n(M)$. 
				\item $\D$ vanishes on multiples of $\omega_s$.
				\item $\bar{d}$ and $\D$ vanish on closed forms.
				\item If $\phi:M\rightarrow M$ is a symplectomorphism, then $\bar{d}$ and $\phi^*$ commute. The same holds for $\D$.
			\end{enumerate}
		\end{proposition}
		\begin{proof}
			\begin{enumerate}
				\item As $\D\tau$ has degree $n$, we only need to show that $\omega_s\wedge \D\tau=0$ due to Proposition \ref{proposition_Lefschetz-decomposition}. This follows from $\omega_s\wedge \D\tau=\omega_s\wedge dL^{-1}d\tau=d(\omega_s\wedge L^{-1}d\tau)=d^2\tau=0$, as $\omega_s$ is closed.
				\item If $\tau=\omega_s\wedge\xi$, then $\D\tau=dL^{-1}d(\omega_s\wedge\xi)=dL^{-1}(\omega_s\wedge d\xi)=d(d\xi)=0$.
				\item Trivial.
				\item We calculate  $\omega_s\wedge\bar{d}(\phi^*\tau)=d(\phi^*\tau)=\phi^*d\tau=\phi^*(\omega_s\wedge\bar{d}\tau)=\omega_s\wedge \phi^*\bar{d}\tau$. By dividing by $\omega_s$, we obtain $\bar{d}(\phi^*\tau)=\phi^*\bar{d}\tau$. $\D(\phi^*\tau)=\phi^*\D\tau$ follows by applying $d$ to both sides.
			\end{enumerate}
		\end{proof}
		For a Euclidean vector space $V$ consider the symplectic vector space $V\times V$ with symplectic form $\omega_s((v_1,v_2),(w_1,w_2):=\langle v_1,w_2\rangle -\langle v_2,w_1\rangle$. An isotropic subspace $W\subset V\times  V$ is called strictly positive if there exists $k$ orthogonal vectors $u_1,\dots,u_k\in V$ such that $W$ is spanned by the vectors $w_i:=(u_i,\lambda_i u_i)$ where $\lambda_i>0$ for all $1\le i\le k$.
		We will need the following lemma, which is due to Bernig. It is an easy generalization of \cite{Bernig_Broecker:Valuations_manifolds_Rumin_cohomology} Lemma 1.4.
		\begin{lemma}
			\label{lemma_positive_isotropic_subspace}
			If $\tau\in\Omega^k(V\times V)$ vanishes on all strictly positive isotropic $k$-dimensional subspaces, then $\tau$ is a multiple of the symplectic form.
		\end{lemma}
		In addition, recall that the natural $1$-form $\alpha\in \Omega^1(T^*V)$ is defined by $\alpha|_{x,y}=\langle y,d\pi|_x\cdot\rangle$ for $(x,y)\in T^*V$. If we choose linear coordinates $(x_1,\dots,x_n)$ on $V$ with induced coordinates $(y_1,\dots,y_n)$ on $V^*$, this form is given by $\alpha=\sum_{i=1}^{n}y_idx_i$. In particular, $-d\alpha=\omega_s$.
		\begin{proof}[Proof of Theorem \ref{maintheorem_kernel_theorem}]
			Let $\mu:=D(\cdot)[\tau]$ be the valuation induced by $\tau$.\\
			Let us start by showing that $\D\tau=0$ and $\int_V\tau=0$ imply $\mu=0$. By continuity (see Corollary \ref{corollary_Continuity_valuations_defined_by_differential_cycle}), it is enough to show $\mu(f)=0$ for all $f\in\Conv(V,\R)\cap C^\infty(V)$. Set $\xi:=\bar{d}\tau$, i.e. $\omega_s\wedge\xi=d\tau$. As $\omega_s=-d\alpha$ for the natural $1$-form on $T^*V$, $d(-\alpha\wedge\xi)=\omega_s\wedge\xi+\alpha\wedge d\xi=d\tau+\alpha\wedge \D\tau$. Thus $\D\tau=0$ implies $d(\tau+\alpha\wedge \xi)=0$.\\
			
			Consider the valuation $f\mapsto D(f)[\alpha\wedge\xi]$. As $f$ is smooth, it is easy to see that $D(f)\llcorner\alpha=D(f)\llcorner \pi^*df$, so
			\begin{align*}
			D(f)[\alpha\wedge\xi]=&D(f)[d(\pi^*f)\wedge \xi]=D(f)[d(\pi^*f\wedge \xi)-\pi^*f\wedge d\xi]\\
			=&-D(f)[\pi^*f\wedge \D\tau]=0.
			\end{align*}
			Here we have used that $D(f)$ is closed. Now observe that $D(f)$ and $[V\times\{0\}]=D(0)$ belong to the same homology class. Thus $d(\tau+\alpha\wedge\xi)=0$ and $D(f)[\alpha\wedge\xi]=0$ imply
			\begin{align*}
			D(f)[\tau]=D(f)[\tau+\alpha\wedge\xi]=D(0)[\tau+\alpha\wedge\xi]=\int_V(\tau+\alpha\wedge\xi)=\int_V\tau=0,
			\end{align*}
			as $\alpha|_V=0$. Thus $\mu(f)=0$.\\
			
			Now let us assume that $\mu=0$. Because $\int_V\tau=D(0)[\tau]=\mu(0)=0$, the second condition follows directly.\\
			Let $f$ be a smooth convex function such that $f-\lambda|\cdot|^2$ is convex for some $\lambda>0$. Then $D(f)$ is given by integration over the graph of $df$ and for $g\in C^\infty_c(V)$ there exists $\epsilon>0$ such that the function $f+tg$ is convex for $t\in (-\epsilon,\epsilon)$. Proposition \ref{proposition_Fu_sum_MA_andC11} thus shows that
			\begin{align*}
			0=D(f+tg)[\tau]=\Phi_{tg*}D(f)[\tau]=D(f)[\Phi_{tg}^*\tau]\quad \text{for }t\in(-\epsilon,\epsilon),
			\end{align*}
			where $\Phi_{tg}:T^*V\rightarrow T^*V$, $\Phi_{tg}(x,y)=(x,y+tdg(x))$. Differentiating, we obtain
			\begin{align*}
			0=D(f)[\mathcal{L}_{X_g}\tau]=D(f)[(d\circ i_{X_g}+i_{X_g}\circ	 d)\tau]=D(f)[i_{X_g}d\tau],
			\end{align*}
			as $D(f)$ is closed. Here $X_g:= \frac{d}{dt}\big|_0\Phi_{tg}$. Using $d\tau=\omega_s\wedge \bar{d}\tau$, 
			\begin{align*}
			0=D(f)[i_{X_g}(\omega_s\wedge \bar{d}\tau)]=D(f)[\omega_s\wedge i_{X_g}\bar{d}\tau+i_{X_g}\omega_s\wedge \bar{d}\tau].
			\end{align*}
			As $D(f)$ is Lagrangian, the first term vanishes, so we are left with
			\begin{align*}
			0=D(f)[i_{X_g}\omega_s\wedge \bar{d}\tau].
			\end{align*}
			The map $\Phi_{tg}$ is a symplectomorphism and it is easy to see that $\Phi_{tg}^*\alpha=\alpha+t\pi^*dg$. Differentiating, we obtain $\pi^*dg=\mathcal{L}_{X_g}\alpha=i_{X_g}d\alpha+di_{X_g}\alpha=-i_{X_g}\omega_s$. Here we have used that $d\pi (X_g)=0$, i.e. $i_{X_g}\alpha=0$, as $\Phi_{tg}$ maps each fiber to itself.	In particular,
			\begin{align*}
			0=D(f)[d(\pi^*g)\wedge\bar{d}\tau].
			\end{align*}
			Using once again that $D(f)$ is closed, we arrive at 
			\begin{align*}
			0=D(f)[\pi^*g\wedge d\bar{d}\tau]=D(f)[\pi^*g\wedge \D\tau].
			\end{align*}
			As this is true for all $g\in C^\infty_c(V)$, $\D\tau$ vanishes on all spaces tangent to the graph of $df$.\\
			We are now going to apply Lemma \ref{lemma_positive_isotropic_subspace}: Fix a Euclidean structure on $V$ and use the induced isomorphism $V\times V^*\cong V\times V$. It can be checked that this is a symplectomorphism. Fix a point $(x,y)\in T^*V$. We claim that the pullback of $\D\tau$ vanishes on all strictly positive isotropic $n$-dimensional subspaces at the corresponding point in $V\times V$. Given a strictly positive isotropic subspace $W$ of dimension $n$, we thus need to find a convex function $f\in\Conv(V,\R)\cap C^\infty(V)$ such that $f-\lambda|\cdot|^2$ is convex for some $\lambda>0$, $df(x)=y$ and such that the tangent space to the graph of $\nabla f$ is exactly $W$. By definition, there exist orthonormal vectors $u_1,\dots,u_n\in V$ and positive numbers $\lambda_1,\dots,\lambda_n$ such that $W$ is spanned by $w_i=(u_i,\lambda_i u_i)$. With respect to the basis $u_1,\dots,u_n$, we obtain linear coordinates $z_1,\dots, z_n$ on $V$, and we define $f\in\Conv(V,\R)\cap C^\infty(V)$ by
			\begin{align*}
			f(z):=\sum_{i=1}^{n}\frac{1}{2}\lambda_i z_i^2+(y_i-\lambda_ix_i) z_i,
			\end{align*}
			where $(y_1,\dots,y_n)$ are the coordinates  with respect to the basis $u_1,\dots,u_n$ of the image of $y\in V^*$ in $V$ under the isomorphism above. Then $f$ has the desired properties. Using that $\D\tau$ vanishes on all spaces tangent to the graph of $df$, we see that the pullback of $\D\tau$ to $V\times V$ vanishes on $W$. As this is true for all strictly positive isotropic $n$-dimensional subspaces and all $(x,y)\in T^*V\cong V\times V$, this pullback must be a multiple of the symplectic form due to Lemma \ref{lemma_positive_isotropic_subspace}. As $T^*V\cong V\times V$ are symplectomorphic, $\D\tau$ must be a multiple of the symplectic form on $T^*V$ as well. However, $\D\tau$ is primitive due to Proposition \ref{proposition_properties_tilde_D_d}, so the Lefschetz decomposition in Proposition \ref{proposition_Lefschetz-decomposition} implies $\D\tau=0$.
		\end{proof}
		
	\begin{corollary}
		\label{corollary_Dtau=0_euler_characteristic}
		If $\tau\in\Omega^n_{hc}(T^*V)$ satisfies $\D\tau=0$, then $D(f)[\tau]=\int_V\tau$ for all $f$.
	\end{corollary}
	\begin{proof}
		Choose $\phi\in C^\infty_c(V)$ such that $\int_V\tau=\int_V\phi(x) d\vol(x)$. Then $\D(\pi^*(\phi\wedge\vol))=0$, as $d\pi^*(\phi\wedge\vol)=0$. By definition $\int_V\left(\tau-\pi^*(\phi\wedge\vol)\right)=0$, so the valuations induced by $\tau$ and $\pi^*(\phi\wedge\vol)$ have to coincide by Theorem \ref{maintheorem_kernel_theorem}. But $D(f)[\pi^*(\phi\wedge\vol)]=\int_V\phi(x)d\vol(x)=\int_V\tau$ by the defining properties of the differential cycle.
	\end{proof}

\section{Invariant smooth valuations}
	\label{section:invariant_smooth_valuations}
	\subsection{Representation of dually epi-translation invariant valuations by invariant forms}
	\label{section:representation_smooth_dually_epi}
	We will call a differential form on $T^*V$ vertically translation invariant if it is invariant with respect to translations in the second component of $T^*V=V\times V^*$.	
	\begin{corollary}
		\label{corollary:smooth_dually_Dtau}
		A differential form $\tau$ represents a dually epi-translation invariant valuation $\mu$ if and only if $\D\tau$ is vertically translation invariant and $\int_V\phi_\lambda^*\tau=\int_V\tau$ for all $\lambda\in V^*$, where $\phi_\lambda:T^*V\rightarrow T^*V$, $\phi(x,y)=(x,y+\lambda)$.
	\end{corollary}
	\begin{proof}
		By Proposition \ref{proposition_Fu_sum_MA_andC11} $D(f+\lambda)[\tau]=\phi_{\lambda*}D(f)[\tau]=D(f)[\phi_\lambda^*\tau]$ for all $\lambda\in V^*$. If $\mu$ is dually epi-translation invariant, this implies that $\tau$ and $\phi_{\lambda}^*\tau$ induce the same valuation $\mu$. Theorem \ref{maintheorem_kernel_theorem} shows $\D(\tau-\phi_\lambda^*\tau)=0$ and $\int_V\phi_\lambda^*\tau=\int_V\tau$. But it is easy to see that $\phi_\lambda$ is a symplectomorphism, so $\D\tau=\phi_\lambda^*\D\tau$ by Proposition \ref{proposition_properties_tilde_D_d} for all $\lambda\in V^*$, i.e. $\D\tau$ is vertically translation invariant.\\
		Now if $\D\tau$ is vertically translation invariant, then $\D\tau=\phi_\lambda^*\D\tau=\D(\phi_\lambda^*\tau)$ for all $\lambda\in V^*$ as $\phi_\lambda$ is a symplectomorphism. Together with the second property, Theorem \ref{maintheorem_kernel_theorem} implies that $\tau$ and $\phi_\lambda^*\tau$ induce the same valuation, so
		\begin{align*}
			D(f)[\tau]=D(f)[\phi_\lambda^*\tau]=(\phi_\lambda)_*D(f)[\tau]=D(f+\lambda)[\tau]\quad\forall f\in\Conv(V,\R),\forall \lambda\in V^*
		\end{align*}
		by Proposition \ref{proposition_Fu_sum_MA_andC11}. Of course, any valuation obtained from the differential cycle is invariant under the addition of constants, so $\tau$ induces a dually epi-translation invariant valuation.
	\end{proof}
	Consider the map $m_t:T^*V\rightarrow T^*V$, $(x,y)\mapsto (x,ty)$ for $t>0$. We will call a differential form $\tau$ on $T^*V$ homogeneous of degree $k\in \R$ if $m_t^*\tau=t^k\tau$ for all $t>0$.
	\begin{corollary}
		\label{corollary_homogeneous_valuations_and_D}
		For $k\ge 0$, $\tau$ represents a $k$-homogeneous valuation $\mu$ if and only if
		\begin{enumerate}
			\item $\D\tau$ is $(k-1)$-homogeneous and $\int_Vm_t^*\tau=0$ for all $t>0$ if $k\ne 0$,
			\item $\D\tau=0$ if $k=0$. 
		\end{enumerate}
		In particular, $\tau$ induces a constant valuation if and only if $\D\tau=0$.
	\end{corollary}
	\begin{proof}
		Using Equation \eqref{equation_multiples_MA_functions}, $D(tf)=m_{t*}D(f)$.
		$m_t$ is not a symplectomorphism but $m_t^*\omega_s=t\omega_s$, i.e. we obtain a constant multiple of $\omega_s$. Then
		\begin{align*}
		d(m_t^*\tau)=m^*_td\tau=m^*_t(\omega_s\wedge\bar{d}\tau)=t\omega_s\wedge m_t^*\bar{d}\tau,
		\end{align*}
		and thus $\bar{d}(m_t^*\tau)=tm_t^*\xi=tm_t^*\bar{d}\tau$ and $\D(m^*_t\tau)=d(tm_t^*\bar{d}\tau)=tm_t^*\D\tau$.\\
		Let $\mu$ be $k$-homogeneous. Then
		\begin{align*}
		D(f)[m_t^*\tau]=\mu(tf)=t^k\mu(f)=D(f)[t^k\tau]\quad \text{for all }f\in\Conv(V,\R).
		\end{align*}
		Theorem \ref{maintheorem_kernel_theorem} implies $\D(m_t^*\tau-t^k\tau)=0$ and using the computation above, we obtain $tm_t^*\D\tau=t^k\D\tau$ for all $t>0$.\\
		If $k=0$, we obtain $\D\tau=0$ by considering the limit $t\rightarrow0$. If $k\ne0$, we can divide by $t$ to obtain $m_t^*\D\tau=t^{k-1}\D\tau$ for all $t>0$, i.e. $\D\tau$ is $(k-1)$-homogeneous. Obviously, $\int_Vm_t^*\tau=D(0)[m_t^*\tau]=m_{t*}D(0)[\tau]=D(t\cdot0)[\tau]=\mu(0)=0^k\cdot\mu(0)=0$ if $k>0$.\\
		Now assume that $\D\tau$ is $k-1$ homogeneous, $k\ne 0$, and $\int_Vm_t^*\tau=0$ for all $t>0$. With the same computation as before, we conclude that $\D(m^*_t\tau-t^k\tau)=0$ for all $t>0$. As $\int_Vm_t^*\tau=0$ by assumption, $m^*_t\tau$ and $t^k\tau$ induce the same valuation by Theorem \ref{maintheorem_kernel_theorem}, i.e. $\mu(tf)=t^k\mu(f)$ for all $t>0$.\\
		If $\D\tau=0$, then $\mu$ is constant by Corollary \ref{corollary_Dtau=0_euler_characteristic} and in particular $0$-homogeneous.
	\end{proof}
		Let us also make the following observations:
	\begin{lemma}
		\label{lemma_alternative_form_valuation_with_D}
		Let $\mu\in\VConv(V)$ be a smooth valuation induced by $\tau\in\Omega_{hc}^n(T^*V)$. For every function $\psi\in C^\infty(V)$ (not necessarily with compact support)
		\begin{align*}
		\frac{d}{dt}\Big|_0\mu(f+t\psi)=D(f)[\pi^*\psi\wedge \D\tau]\quad\forall f\in\Conv(V,\R).
		\end{align*}
		In particular, the left hand side of this equation defines a smooth valuation. If $\mu\in\VConv_k(V)$, $k\ne 0$, this implies that
		\begin{align*}
		D(f)[\tau]=\frac{1}{k}D(f)\left[\pi^*f\wedge \D\tau\right]
		\end{align*}
		for all $f\in\Conv(V,\R)\cap C^\infty(V)$. Moreover for $k=1$, there exists $\phi\in C^\infty_c(V)$ such that
		\begin{align*}
		D(f)[\tau]=\int_{V}f(x)\phi(x)d\vol(x)
		\end{align*}
		for all $f\in\Conv(V,\R)$.
	\end{lemma}
	\begin{proof}
		First note that every smooth valuation naturally extends to a functional on all Monge-Amp\`ere functions, so the left hand side of the first equation is well defined for all $f\in\Conv(V,\R)$ by Proposition \ref{proposition_Fu_sum_MA_andC11}. Note that
		\begin{align*}
		\frac{d}{dt}\Big|_0\mu(f+t\psi)=\frac{d}{dt}\Big|_0D(f+t\psi)[\tau]=\frac{d}{dt}\Big|_0D(f)[G_{t\psi}^*\tau]
		\end{align*}
		for $G_{t\psi}(x,y):=(x,y+td\psi(x))$ by Proposition \ref{proposition_Fu_sum_MA_andC11}. Denoting $X_\psi:=\frac{d}{dt}|_0 G_{t\psi}$, we see that
		\begin{align*}
		\frac{d}{dt}\Big|_0\mu(f+t\psi)=&D(f)\left[\mathcal{L}_{X_\psi}\tau\right]=D(f)\left[i_{X_\psi}d\tau\right]=D(f)\left[i_{X_\psi}(\omega_s\wedge \bar{d}\tau)\right]\\
		=& D(f)\left[i_{X_\psi}\omega_s\wedge \bar{d}\tau\right]=-D(f)\left[d\pi^*\psi\wedge\bar{d}\tau\right]=D(f)\left[\pi^*\psi\wedge \D\tau\right].
		\end{align*}
		Here we have used that $i_{X_\psi}\omega_s=-\pi^*d\psi$ as in the proof of Theorem \ref{maintheorem_kernel_theorem}.\\

		Now assume that $\mu$ is $k$-homogeneous, $f\in\Conv(V,\R)\cap C^\infty(V)$. Applying the previous argument to $\psi=f$, we obtain
		\begin{align*}
			kD(f)[\tau]=\frac{d}{dt}\Big|_0(1+t)^kD(f)[\tau]=\frac{d}{dt}\Big|_0D(f+tf)[\tau]=D(f)[\pi^*f\wedge\D\tau].
		\end{align*}
		
		For $k=1$, $\D\tau$ is a vertically translation invariant form that is in addition homogeneous of degree $0$ according to Corollary \ref{corollary:smooth_dually_Dtau} and Corollary \ref{corollary_homogeneous_valuations_and_D}. Thus $\D\tau=\pi^*(\phi\vol)$ for some $\phi\in C^\infty_c(V)$, so
		\begin{align*}
			D(f)[\tau]=D(f)[\pi^*f\wedge\D\tau]=D(f)[\pi^*(f\cdot\phi\vol)]=\int_V f(x)\phi(x)d\vol(x)
		\end{align*}
		for all $f\in\Conv(V,\R)\cap C^\infty(V)$ due to the defining properties of the differential cycle. By continuity, this equation thus holds for all $f\in\Conv(V,\R)$.
	\end{proof}
	
	In the rest of this section, we will show that any dually epi-translation invariant valuation, which can be represented by some (not necessarily invariant) differential form, can actually be obtained from a vertically translation invariant differential form.\\
	Let $\Omega^{k,l}=\Omega_c^k(V)\otimes\Lambda^l(V^*)^*\subset\Omega_{hc}^{k+l}(T^*V)$ denote the space of differential forms of bidegree $(k,l)$ with horizontally compact support that are in addition vertically translation invariant. The next proposition describes the image of $\D:\Omega^{n-k,k}\rightarrow\Omega^{n-(k-1),k-1}$. Recall that the de Rham cohomology with compact support is given by 
	\begin{align*}
		H_c^k(\R^n)\cong\begin{cases}
		0 & k\ne n,\\
		\R & k=n.
		\end{cases}
	\end{align*} This isomorphism is realized by the map $[\tau]\mapsto \int_{\R^n}\tau$.
	\begin{proposition}
		\label{proposition_desription_image_symp_D}
		For $2\le k\le n$:
		\begin{align*}
			\mathrm{Im}(\D:\Omega^{n-k,k}\rightarrow\Omega^{n-(k-1),k-1})=\ker d\cap \ker L\cap \Omega^{n-(k-1),k-1}.
		\end{align*}
		For $k=1$:
		\begin{align*}
			&\mathrm{Im}(\D:\Omega^{n-1,1}\rightarrow\Omega^{n,0})\\
			=&\left\{\pi^*(\phi\wedge\vol):\phi\in C^\infty_c(V), \int_V\phi(x) d\vol(x)=\int_V\lambda(x)\phi(x) d\vol(x)=0\quad\forall \lambda\in V^*\right\}.
		\end{align*}
	\end{proposition}
	\begin{proof}
		To keep the notation simple, let us suppress the pullbacks along the projections from $T^*V=V\times V^*$ onto the two factors.\\
		Let us start with the case $2\le k\le n$. Examining the degrees and using Corollary \ref{corollary:smooth_dually_Dtau} as well as Proposition \ref{proposition_properties_tilde_D_d}, we see that the image of $\D$ is contained in the space on the right. For the converse, let $\tau\in \ker d\cap \ker L\cap \Omega^{n-(k-1),k-1}$ be given. Choosing a basis $\xi_i$, $1\le i\le \binom{n}{k-1}$, of $\Lambda^{k-1}(V^*)^*$, we find differential forms $\phi_i\in\Omega_c^{n-(k-1)}(V)$ such that
		\begin{align*}
		\tau=\sum\limits_i\xi_i\wedge\phi_i.
		\end{align*}
		As $\tau$ is closed, $0=(-1)^{k-1}\sum_i\xi_i\wedge d\phi_i$, and thus $d\phi_i=0$ for all $i$. Using $H_c^{n-(k-1)}(V)=0$ for $2\le k\le n$, we see that there exists $\psi_i\in \Omega_c^{n-k}(V)$ such that $\phi_i=(-1)^{k-1}d\psi_i$. Set 
		\begin{align*}
		\omega:=\sum\limits_i\xi_i\wedge \psi_i,
		\end{align*}
		i.e. $\tau=d\omega$. Then $\omega_s\wedge\omega$ is closed, as $d(\omega_s\wedge\omega)=\omega_s\wedge d\omega=\omega_s\wedge \tau=0$, because $\tau$ belongs to the kernel of $L$. We will need to find a vertically translation invariant $n$-form $\tilde{\tau}$ such that $d\tilde{\tau}=\omega_s\wedge\omega$.\\
		Note that $\omega_s\wedge \omega\in\Omega^{n+1}(T^*V)$ is again a vertically translation invariant differential form, now of bidegree $(n-k+1,k)$. If $\tilde{\xi}_i$, $i=1,\dots,\binom{n}{k}$, denotes a basis of $\Lambda^{k}(V^*)^*$, there exist unique differential forms $\tilde{\phi}_i\in\Omega^{n-k+1}_c(V)$ such that 
		\begin{align*}
		\omega_s\wedge\omega=\sum\limits_i\tilde{\xi}_i\wedge\tilde{\phi}_i.
		\end{align*}
		As $\omega_s\wedge\omega$ is closed, we obtain $0=(-1)^k\sum_i\tilde{\xi}_i\wedge d\tilde{\phi}_i$, i.e. $d\tilde{\phi}_i=0$ for all $i$. Using $H_c^{n-k+1}(V)=0$ for $k\ne 1$, we obtain $\tilde{\psi}_i\in\Omega^{n-k}_c(V)$ such that $\tilde{\phi}_i=(-1)^kd\tilde{\psi}_i$. Then $\omega_s\wedge\omega=d\tilde{\tau}$, where
		\begin{align*}
		\tilde{\tau}=\sum\limits_i\tilde{\xi}_i\wedge\tilde{\psi_i}\in\Omega^{n-k,k}.
		\end{align*}
		Thus $\D\tilde{\tau}=d\omega=\tau$.\\
		
		For $k=1$, $\omega\in \Omega^{n-1,n}$ defines a $1$-homogeneous valuation in $\VConv(V)$, and $\D\omega\in \Omega^{n,0}$ is given by $\phi\wedge\vol$ for some $\phi\in C^\infty_c(V)$. Lemma \ref{lemma_alternative_form_valuation_with_D} implies
		\begin{align*}
			D(f)[\omega]=\int_V f(x)\phi(x) d\vol(x)\quad\forall f\in\Conv(V,\R).
		\end{align*}
		As $\omega\in\Omega^{n-1,1}$, it induces a dually epi-translation invariant valuation, so the left hand side of this equation vanishes for affine functions. Thus $\mathrm{Im}(\D:\Omega^{n-1,1}\rightarrow\Omega^{n,0})$ is contained in
		\begin{align*}
			\left\{\phi\wedge\vol:\phi\in C^\infty_c(V), \int_V\phi(x) d\vol(x)=\int_V\lambda(x)\phi(x) d\vol(x)=0\quad\forall \lambda\in V^*\right\}.
		\end{align*}
		
		For the converse direction, choose oriented orthonormal coordinates $x_1,\dots,x_n$ on $V$ with induced coordinates $y_1,\dots,y_n$ on $V^*$, and let $\phi\in C^\infty_c(V)$ be a function with $\int_V\phi(x) d\vol(x)=0$ and $\int_Vx_i\phi(x)d\vol(x)=0$ for $1\le i\le n$. Then $\phi\wedge\vol$ belongs to the trivial cohomology class of $H^n_c(V)\cong \R$, so there exists $\omega\in \Omega^{n-1}_c(V)$ such that $\phi\wedge \vol=d\omega$. Now $-\omega_s\wedge\omega=d\alpha\wedge\omega=d(\alpha\wedge\omega)-\alpha\wedge d\omega$. As $d\omega$ is a multiple of $\vol$, the second term vanishes, and we are left with $d\alpha\wedge \omega=d(\alpha\wedge\omega)$. In other words, $\bar{d}(\alpha\wedge\omega)=-\omega$. Thus $\D(-\alpha\wedge\omega)=d\omega=\phi\wedge\vol$. 
		Note that Lemma \ref{lemma_alternative_form_valuation_with_D} implies
		\begin{align*}
		D(f)[-\alpha\wedge\omega]=D(f)\left[f\D(-\alpha\wedge\omega)\right]=D(f)\left[f\phi\wedge \vol\right]=\int_V f(x)\phi(x)d\vol(x)
		\end{align*}
		for all $f\in\Conv(V,\R)\cap C^\infty(V)$. 
		As $\int_V\phi(x)d\vol(x)=\int_Vx_i\phi(x)d\vol(x)=0$ for all $1\le i\le n$, the valuation induced by $-\alpha\wedge\omega$ is thus dually epi-translation invariant. On the other hand, $\alpha\wedge\omega=\sum_{i=1}^{n}y_i\phi_i\vol$ for some $\phi_i\in C^\infty_c(V)$. The defining properties of the differential cycle imply
		\begin{align*}
		D(f)[-\alpha\wedge\omega]=-D(f)\left[\sum_{i=1}^{n}y_i\phi_i\vol\right]=-\sum_{i=1}^{n}\int_V \partial_i f(x)\phi_i(x)d\vol(x).
		\end{align*}
		Taking $f(x)=x_i$, we thus deduce  $0=\int_V \phi_i(x) d\vol(x)$ for all $i=1,\dots,n$, i.e. $\phi_i\wedge \vol$ is trivial in cohomology. Thus we can find $\psi_i\in \Omega^{n-1}_c(V)$ such that $d\psi_i=\phi_i\wedge\vol$. In total, $\alpha\wedge\omega=\sum_{i=1}^{n}y_id\psi_i=d(\sum_{i=1}^{n}y_i\psi_i)-\sum_{i=1}^{n}dy_i\wedge \psi_i$. Then $\tau:=\sum_{i=1}^{n}dy_i\wedge\psi_i\in\Omega^{n-1,1}$ satisfies \begin{align*}
			\D\tau=\D(-\alpha\wedge\omega+d(\sum_{i=1}^{n}y_i\psi_i))=-\D(\alpha\wedge\omega)=\phi\wedge\vol.
		\end{align*} Thus $\phi\wedge\vol$ is contained in the image of $\D:\Omega^{n-1}\rightarrow\Omega^{n,0}$.
	\end{proof}
	\begin{theorem}
		\label{theorem:Characterization_smooth_dually_epi_valuations}
		Let $\VConv_k(V)^{sm}\subset\VConv_k(V)$ denote the space of all $k$-homogeneous dually epi-translation invariant valuations of the form $f\mapsto D(f)[\tau]$ for some $\tau\in\Omega_{hc}^n(T^*V)$. Then the following holds:
		\begin{enumerate}
			\item The map $\Omega^{n-k,k}\rightarrow\VConv_k(V)^{sm}$, $\tau\mapsto D(\cdot)[\tau]$ is surjective for all $0\le k\le n$.
			\item For $2\le k\le n$, $\D$ induces an isomorphism
				\begin{align*}
					\VConv_k(V)^{sm}\cong& \mathrm{Im}(\D:\Omega^{n-k,k}\rightarrow\Omega^{n-(k-1),k-1})\\
					=&\ker d\cap \ker L\cap \Omega^{n-(k-1),k-1}.
				\end{align*}
			\item For $k=1$, $\D$ induces an isomorphism
			\begin{align*}
				&\VConv_1(V)^{sm}\cong 	\mathrm{Im}(\D:\Omega^{n-1,1}\rightarrow\Omega^{n,0})\\
				=&\left\{\pi^*(\phi\wedge\vol):\phi\in C^\infty_c(V), \int_V\phi(x) d\vol(x)=\int_V\lambda(x)\phi(x) d\vol(x)=0\quad \forall \lambda\in V^*\right\}.
			\end{align*}
		\end{enumerate}
	\end{theorem}
	\begin{proof}
		As any $k$-homogeneous valuation of degree $k>0$ vanishes in $0\in\Conv(V,\R)$, a valuation $\mu\in\VConv_k(V)^{sm}$ is uniquely determined by $\D\tau$, where $\tau$ is any smooth differential form representing $\mu$, due to Theorem \ref{maintheorem_kernel_theorem}. Thus 2. and 3. follow from 1. using Proposition \ref{proposition_desription_image_symp_D}.
		For $k=0$, the map in 1. is obviously surjective. Thus let $k>0$. Any $k$-homogeneous valuation represented by $\tau\in\Omega_{hc}^n(T^*V)$ vanishes in $0$, i.e. it satisfies $\int_V\tau=D(0)[\tau]=0$, and $\D\tau$ is vertically translation invariant as well $(k-1)$-homogeneous by Corollary \ref{corollary:smooth_dually_Dtau} and Corollary \ref{corollary_homogeneous_valuations_and_D}. For $k\ge 2$, this implies that $\D\tau$ belongs to the image of $\D:\Omega^{n-k,k}\rightarrow\Omega^{n-(k-1),k-1}$ by Proposition \ref{proposition_desription_image_symp_D}, so we find some $\tilde{\tau}\in\Omega^{n-k,k}$ with $\D(\tau-\tilde{\tau})=0$. Of course, any such differential form satisfies $\int_V\tilde{\tau}=0$, so Theorem \ref{maintheorem_kernel_theorem} implies that $\tilde{\tau}$ and $\tau$ induce the same valuation.\\
		For $k=1$, we need to show that $\D\tau=\pi^*(\phi\wedge \vol)$ is in the image of $\D:\Omega^{n-1,1}\rightarrow\Omega^ {n,0}$. Using Proposition \ref{proposition_desription_image_symp_D}, this follows from the fact that $D(\cdot)[\tau]$ is dually epi-translation invariant together with Lemma \ref{lemma_alternative_form_valuation_with_D}. With the same argument as before we find $\tilde{\tau}\in \Omega^{n-1,1}$ with $\D(\tau-\tilde{\tau})=0$ and $\int_V\tilde{\tau}=0=\int_V\tau$. Applying Theorem \ref{maintheorem_kernel_theorem} again, we obtain the desired result.
	\end{proof}

\subsection{Invariance under subgroups of the general linear group}
\label{section:Invariance_under_subgroup}
\begin{proposition}
	\label{proposition_invarant_valuations_and_D}
	Let $G\subset \GL(V)$ be a subgroup. Then $\tau\in\Omega^{n}_{hc}(T^*V)$ induces a $G$-invariant valuation if and only if
	\begin{enumerate}
		\item $g^*\D\tau=\sign(\det g)\D\tau$ for all $g\in G$,
		\item $\int_V\tau=\sign(\det g)\int_V (g^{-1})^*\tau$ for all $g\in G$.
	\end{enumerate}
\end{proposition}
\begin{proof}
	By Proposition \ref{proposition_Fu_differential_cycle_and_diffeomorphisms}, $D(f\circ g)=\sign(\det g)(g^{-1})_*D(f)$, where $g\in G$ operates on $T^*V$ by $g(x,y)=(gx, y\circ g^{-1})$, which is a symplectomorphism. Thus $f\mapsto \mu(f\circ g)$ is represented by the differential form $\sign(\det g)(g^{-1})^*\tau$. In particular, $\tau$ induces a $G$-invariant valuation if and only if $\tau$ and $\sign(\det g)(g^{-1})^*\tau$ induce the same valuation for all $g\in G$, which using Theorem \ref{maintheorem_kernel_theorem} is equivalent to 
	\begin{enumerate}
		\item $\D\tau=\sign(\det g)\D((g^{-1})^*\tau)$ for all $g\in G$,
		\item $\int_V\tau=\sign(\det g)\int_V (g^{-1})^*\tau$ for all $g\in G$.
	\end{enumerate}
	However, Proposition \ref{proposition_properties_tilde_D_d} implies $\D((g^{-1})^*\tau)=(g^{-1})^*\D\tau$ for $g\in G$, so the first condition is equivalent to $g^*\D\tau=\sign(\det(g))\D\tau$ for all $g\in G$. The claim follows.
\end{proof}
\begin{corollary}
	\label{corollary:1-homogeneous_invariant-valuation}
	Let $G\subset\mathrm{O}(n)$ be a closed subgroup of the orthogonal group that operates transitively on the unit sphere in $\R^n$, and let $\mu\in\VConv_1(\R^n)^{sm}$ be a smooth $G$-invariant valuation. Then $\mu$ is $\mathrm{O}(n)$-invariant.
\end{corollary}
\begin{proof}
	Let $\tau\in\Omega^{n-1,1}$ be a differential form inducing $\mu$. By Proposition \ref{proposition_invarant_valuations_and_D}, $\D\tau$ satisfied $g^*\D\tau=\det(g)\D\tau$. Moreover, $\D\tau=\pi^*(\phi\wedge\vol)$ for some $\phi\in C^\infty_c(\R^n)$ by Proposition \ref{proposition_desription_image_symp_D}. As $G$-operates transitively on the unit sphere, $\phi$ is rotation invariant. Moreover, Lemma \ref{lemma_alternative_form_valuation_with_D} implies
	\begin{align*}
		D(f)[\tau]=\int_{\R^n}f(x)\phi(x)d\vol(x)\quad\forall f\in\Conv(\R^n,\R),
	\end{align*}
	which is obviously $\mathrm{O}(n)$-invariant as $\phi$ is rotation invariant.
\end{proof}

\section{Characterization of smooth valuations}
		\label{section_characterization_of_smooth_valuations}
		Let us choose a scalar product on $V$ with induced scalar products on $V^*\cong V$ and $V\times\R$. In addition, let us fix an orientation on $V$ (and thus $V^*$). If $\vol\in\Lambda^nV$ induces the orientation of $V^*$, we will equip $V^*\times\R$ with the orientation induced by $-dt\wedge \vol$, where $dt$ is the standard coordinate form on $\R$. Let us also choose orthonormal linear coordinates $x_1,\dots,x_n$ on $V$ with induced coordinates $(y_1,\dots,y_n)$ on $V^*$. Consider the  map
		\begin{align*}
		Q:(V^*\times \R)\times S(V\times\R)_-&\rightarrow V\times V^*=T^*V\\
		\left(y,s,(x,t)\right)&\mapsto \left(-\frac{x}{t},y\right).
		\end{align*}
		To simplify the notation, let $E:=V^*\times\R$, such that $Q:SE_-:=E\times S(E^*)_-\rightarrow T^*V$.
		\begin{proposition}
			\label{proposition_connection_normal_cycle_differential_cycle_under_restriction}
			Let $K\in\mathcal{K}(V^*\times\R)$. Then
			\begin{align*}
			Q_*\left[\CNC(K)|_{SE_-}\right]=D\left(h_K(\cdot,-1)\right)
			\end{align*}
		\end{proposition}
		\begin{proof}
			As $\supp \CNC(K)\subset K\times S(E^*)$, $Q$ is proper on the support of $\CNC(K)\big|_{SE_-}$. Now observe that both sides depend continuously on $K$ in the local flat metric topology by Proposition \ref{proposition_continuity_conormal_cycle} and Theorem \ref{theorem_continuity_D_on_convex_functions}. It is thus enough to prove the equation for $K\in\mathcal{K}(E)$ smooth with strictly positive Gauss curvature. In this case, the support function of $K$ is smooth outside of $0$ and
			\begin{align*}
			\CNC(K)=\left(d'h_K\times Id\right)_*\left[S(E^*)\right]
			\end{align*}
			by Lemma \ref{lemma_conormal-cycle-smooth-body}. We therefore need to consider the map
			\begin{align*}
			Q\circ \left(d'h_K\times Id\right):S(E^*)_{-}&\rightarrow V\times V^*\\
			(x,t)&\mapsto \left(-\frac{x}{t},\partial_1h_K(x,t)\right),
			\end{align*}
			where $\partial_1h_K=(\partial_{x_1}h_K,\dots,\partial_{x_n}h_K)$. $h_K$ is $1$-homogeneous, so $\partial_1h_K(x,t)=\partial_1h_K(-\frac{x}{t},-1)=df_K(-\frac{x}{t})$ for $t<0$, where $f_K:=h_K(\cdot,-1)$. Thus
			\begin{align*}
			Q\circ \left(d' h_K\times Id\right)(x,t)=\left(-\frac{x}{t},df_K\left(-\frac{x}{t}\right)\right)
			\end{align*}
			for all $(x,t)\in S(E^*)_-$. The map $S(E^*)_-\rightarrow V$, $(x,t)\mapsto -\frac{x}{t}$ is a diffeomorphism and it is easy to see that it is orientation preserving for our choice of orientation. As $D(h_K(\cdot,-1))$ is given by integration over the graph of $d f_K$, we see that both currents coincide.
		\end{proof}
		We will denote the contact form on $E\times S(E^*)$ by $\alpha_E$, $\omega_E:=-d\alpha_E$. Then $\alpha_{E}=tds+\sum_{i=1}^{n}x_idy_i$ with respect to the coordinates $(y,s,x,t)$ on $V^*\times\R\times S(V\times\R)=E\times S(E^*)$.
		\begin{lemma}
			\label{lemma_replacing_invariant_forms_on SE_by_pullbacks}
			Let $\omega\in\Omega^k(E\times S(E^*))$ be a translation invariant differential form. Then there exists a differential form $\omega'\in\Omega^k(T^*V)$ such that $\omega-Q^*\omega'$ is vertical on $E\times S(E^*)_-$, i.e. a multiple of the contact form $\alpha_{E}$.
		\end{lemma}
		\begin{proof}
			Any translation invariant differential form $\omega$ on $E \times S(E^*)$ can be written as a sum of terms of the form $ds\wedge dy^I\wedge \tau$ or $dy^I\wedge \tau$, where $\tau$ is a form on $S(V\times\R)$ of degree $k-|I|-1$ or $k-|I|$ respectively. As $\alpha_{E}=tds+\sum_{j=1}^{n}x_jdy_j$, we can replace $ds$ by $-\frac{1}{t}\sum_{j=1}^{n}x_jdy_j$ while picking up a multiple of $\alpha_{E}$.\\
			We can thus assume that $\omega$ only consists of terms of the form $dy^I\wedge \tau$ with $\tau\in \Omega^*(S(V\times\R))$, i.e. $\omega$ is the pullback of a form $\tilde{\omega}$ on $ V^*\times S(V\times\R)$. Obviously, $\tilde{Q}:V^*\times S(V\times\R)_-\rightarrow T^*V$, $(y,(x,t))\mapsto (-\frac{x}{t},y)$ is a diffeomorphism, so if we denote by $\tilde{\pi}:V^*\times\R\times S(V\times\R)_-\rightarrow V^*\times S(V\times\R)_-$ the obvious projection, we see that $\tilde{Q}\circ \tilde{\pi}=Q$. The claim follows by setting $\omega':=(\tilde{Q}^{-1})^*\tilde{\omega}$.
		\end{proof}
		Due to Theorem \ref{theorem_kernel_theorem_convex_bodies} and Theorem \ref{maintheorem_kernel_theorem}, a smooth valuation is (up to its $0$-homogeneous component) uniquely defined by the (symplectic) Rumin operator of a representing form. We will thus need the following relation between the two versions of the differential.
		\begin{corollary}
			\label{corollary_connection_D_tilde_D}
			For any smooth differential form $\tau\in\Omega^n(T^*V)$, $DQ^*\tau=-\frac{1}{t}\alpha_E\wedge Q^*\D\tau$.
		\end{corollary}
		\begin{proof}
			Let $\omega_{V}$ denote the symplectic form on $T^*V$. A short calculation shows $Q^*\omega_{V}=\frac{1}{t}\omega_{E}+\frac{1}{t^2}dt\wedge \alpha_{E}$. Let $\xi\in\Omega^{n-1}(T^*V)$ be the unique form with $\omega_{V}	 \wedge\xi=d\tau$. Pulling back this equation, we see that
			\begin{align*}
			dQ^*\tau=Q^*\omega_{V}\wedge Q^*\xi=\frac{1}{t}\omega_E\wedge Q^*\xi+\frac{1}{t^2}dt\wedge \alpha_E\wedge Q^*\xi.
			\end{align*}
			Restricting this equation to the contact distribution $H$ in $E\times S(E^*)$, we obtain 
			\begin{align*}
			dQ^*\tau|_H=\frac{1}{t}\omega_E\wedge Q^*\xi|_H=\omega_E|_H\wedge \frac{1}{t}Q^*\xi|_H.
			\end{align*}
			This implies
			\begin{align*}
				d(Q^*\tau+\alpha_E\wedge \frac{1}{t}Q^*\xi)=&dQ^*\tau-\omega_E\wedge \frac{1}{t}Q^*\xi-\alpha_E\wedge d(\frac{1}{t}Q^*\xi)\\
				=&-\frac{1}{t^2}\alpha_E\wedge dt\wedge Q^*\xi+\frac{1}{t^2}\alpha_E\wedge dt\wedge Q^*\xi-\alpha_E\wedge \frac{1}{t}dQ^*\xi\\
				=&-\frac{1}{t}\alpha_E\wedge Q^*d\xi=-\frac{1}{t}\alpha_E\wedge Q^*\D\tau,
			\end{align*}
			which is vertical. Thus $D(Q^*\tau)=d(Q^*\tau+\alpha_E\wedge \frac{1}{t}Q^*\xi)=-\frac{1}{t}\alpha_E\wedge Q^*\D\tau$. 
		\end{proof}		
		
		\begin{proposition}
			\label{proposition_characterization_smooth_valuations_VConv}
			Let $\mu\in\VConv_k(V)$ be a valuation such that $T(\mu)\in\Val_k(V^*\times\R)^{sm}$. Then there exists a differential form $\tau\in\Omega^{n-k,k}$ such that
			\begin{align*}
				\mu(f)=D(f)[\tau]\quad \forall f\in\Conv(V,\R).
			\end{align*}
			In particular $\mu\in\VConv_k(V)^{sm}$ if and only if $T(\mu)\in\Val_k(V^*\times\R)^{sm}$.
		\end{proposition}
		\begin{proof}
			We may assume that $1\le k\le n$. As $T(\mu)$ is a smooth valuation, it can be represented by a smooth translation invariant differential form $\omega\in E\times S(E^*)$. Using Lemma \ref{lemma_replacing_invariant_forms_on SE_by_pullbacks}, we can find a differential form $\omega'\in\Omega^n(T^*V)$ such that $\omega-Q^*\omega'$ differ by a multiple of $\alpha_E$ on $E\times S(E^*)_-$. Applying the Rumin operator and using Corollary \ref{corollary_connection_D_tilde_D}, we obtain
			\begin{align*}
				D\omega=DQ^*\omega'=-\frac{1}{t}\alpha_E\wedge Q^*\D\omega'=-ds\wedge Q^*D\omega'-\frac{1}{t}\sum_{i=1}^{n}x_idy_i\wedge Q^*\D\omega' \quad \text{on }E\times S(E^*)_-.
			\end{align*}
			Note that $Q^*\D\omega'$ does not contain a multiple of $ds$, so the two terms on the right hand side of this equation are linearly independent.
			By Theorem \ref{theorem_image_VCONV->Val}, the vertical support of $T(\mu)$ is compactly contained in $S(E^*)_-$. From Proposition \ref{proposition_characterization_support_GW_smooth_valuation} we deduce that $D\omega$ has support compactly contained in $E\times S(E^*)_-$, so the same applies to $ds\wedge Q^*\D\omega'$ and thus $Q^*\D\omega'$. 
			Thus the support of $\D\omega'$ is horizontally compact. By construction, this is a vertically translation invariant form of bidegree $(n+1-k,k-1)$. For $2\le k\le n$, we can directly apply Proposition \ref{proposition_desription_image_symp_D} to find a vertically translation invariant form $\tau\in\Omega^{n-k,k}$ such that $\D\tau=\D\omega'$. To apply Proposition \ref{proposition_desription_image_symp_D} for $k=1$, note that $\D\omega'=\pi^*(\phi\wedge\vol)$ for some $\phi\in C^\infty_c(V)$, so we have to show that $\phi$ is orthogonal to affine functions. As $\omega$ induces a translation invariant valuation that is $1$-homogeneous,
			\begin{align*}
				0=\CNC(\{\lambda\})[\omega]=\frac{d}{dt}\Big|_0\CNC(\{0\}+t\{\lambda\})[\omega]=\CNC(\{0\})[h_{\{\lambda\}}(x,t) i_RD\omega]\quad \text{for }\lambda\in E,
			\end{align*}
			where we have used Proposition \ref{proposition:formula_first_variation} in the last step. As support functions are $1$-homogeneous,
			\begin{align*}
				-\frac{1}{t}h_{L}(x,t)=-\sign(t)h_L\left(\frac{x}{|t|},\sign(t)\right)\quad\text{for } (x,t)\in V\times\R=E^*, t\ne 0
			\end{align*}
			for $L\in\mathcal{K}(E)$. Because the conormal cycle vanishes on multiples of $\alpha_E$, 
			\begin{align*}
				0=&\CNC(\{0\})[h_{\{\lambda\}}(x,t) i_RD\omega]=\CNC(\{0\})\left[-\frac{1}{t}h_{\{\lambda\}}(x,t)[Q^*\D\omega'-\alpha_E\wedge i_RQ^*\D\omega']\right]\\
				=&\CNC(\{0\})\left[-\frac{1}{t}h_{\{\lambda\}}(x,t)Q^*\D\omega'\right]=\CNC(\{0\})\left[-\sign(t)\lambda\left(\frac{x}{|t|},\sign(t)\right) Q^*\D\omega'\right].
			\end{align*}
			As the support of $Q^*\D\omega'$ is contained in $SE_-$, we obtain
			\begin{align*}
				0=&\CNC(\{0\})|_{SE_-}\left[-\sign(t)\lambda\left(\frac{x}{|t|},\sign(t)\right) Q^*\D\omega'\right]=\CNC(\{0\})|_{SE_-}\left[\lambda\left(-\frac{x}{t},-1\right) Q^*\D\omega'\right]\\
				=&\CNC(\{0\})|_{SE_-}\left[ Q^*(f_\lambda\D\omega')\right]
			\end{align*}
			for $f_\lambda(x):=\lambda(x,-1)$. Proposition \ref{proposition_connection_normal_cycle_differential_cycle_under_restriction} thus implies
			\begin{align*}
				0=D(0)[\pi^*\lambda(\cdot,-1)D\omega']=D(0)[\pi^*(\lambda(\cdot,-1)\cdot \phi\vol)]=\int_V\lambda(x,-1)\phi(x)d\vol(x).
			\end{align*}
			As this holds for all $\lambda\in E=V^*\times\R$, $\phi$ is orthogonal to every affine function, so Proposition \ref{proposition_desription_image_symp_D} implies that we find can $\tau\in\Omega^{n-1,1}$ such that $\D\tau=\pi^*(\phi\vol)=\D\omega'$.\\
			
			It remains to see that $\mu$ is represented by the differential form $\tau$. Observe that
			\begin{align*}
				D\omega=-\frac{1}{t}\alpha_E\wedge Q^*\D\omega'=-\frac{1}{t}\alpha_E\wedge Q^*\D\tau=D(Q^*\tau)\quad \text{on }E\times S(E^*)_-.
			\end{align*}
			By extending $Q^*\tau$ trivially to $E\times S(E^*)$, we see that this equation holds on the whole space, so Theorem \ref{theorem_kernel_theorem_convex_bodies} implies that $\omega$ and $Q^*\tau$ induces the same valuation (note that the second property in Theorem \ref{theorem_kernel_theorem_convex_bodies} is satisfied as the degree of our valuation is positive). In particular
			\begin{align*}
				T(\mu)[K]=&\CNC(K)[\omega]=\CNC(K)[Q^*\tau]=\left(\CNC(K)\big|_{ SE_-}\right)\left[Q^*\tau\right]\\
				=&D\left(h_K(\cdot,-1)\right)[\tau]=T\left(D(\cdot)[\tau]\right)(K)
			\end{align*}
			for any $K\in\mathcal{K}(V^*\times\R)$, where we have used Proposition \ref{proposition_connection_normal_cycle_differential_cycle_under_restriction}. The injectivity of $T$ implies $\mu=D(\cdot)[\tau]$.\\
			It remains to see that any valuation $\mu\in\VConv(V)^{sm}$ satisfies $T(\mu)\in\Val(V^*\times\R)^{sm}$. This follows directly from Proposition \ref{proposition_connection_normal_cycle_differential_cycle_under_restriction} and the characterization of $\Val(V^*\times\R)^{sm}$ in Theorem \ref{theorem_representation_smooth_val_conv_bodies_using_normal_cycle}.
		\end{proof}
		We will now prove a refinement of Theorem \ref{maintheorem_density_smooth_valuations}.
		\begin{theorem}
			\label{theorem_density_smooth_valuations}
			$\VConv_k(V)^{sm}$ is sequentially dense in $\VConv_k(V)$. More precisely, the following holds: For every compact set $A\subset V$, $\mu\in\VConv_{k,A}(V)$, and every compact neighborhood $B\subset V$ of $A$, there exists a sequence $(\mu_j)_j$ in $\VConv_{k,B}(V)\cap \VConv(V)^{sm}$ such that $(\mu_j)_j$ converges to $\mu$.
		\end{theorem}
		\begin{proof}
			Let $\mu\in \VConv_{k,A}(V)$ be given and consider the following commutative diagram with the map $T:\VConv_k(V)\rightarrow\Val_k(V^*\times\R)$ and diffeomorphism 
			\begin{align*}
				P:V&\rightarrow S(V\times\R)_-\\
				x&\mapsto \frac{1}{\sqrt{1+|x|^2}}(x,-1)
			\end{align*}
			from Section \ref{section_dually_epi-translation_invariant_valuations}:
				\begin{center}
					\begin{tikzpicture}
					\matrix (m) [matrix of math nodes,row sep=3em,column sep=4em,minimum width=2em]
					{
						\VConv_{k,A}(V)& \Val_{k,P(A)}(V^*\times\R) \\
						\VConv_{k,B}(V) & \Val_{k,P(B)}(V^*\times\R) \\};
					\path[-stealth]
					(m-1-1) edge node [left] {$ $} (m-2-1)
					edge node [above] {$T$} (m-1-2)
					(m-2-1.east|-m-2-2) edge node [above] {$T$} (m-2-2)
					(m-1-2) edge node [right] {$ $} (m-2-2);
					\end{tikzpicture}
				\end{center}
			The vertical maps are the natural inclusions, while the horizontal maps are topological isomorphisms due to Theorem \ref{theorem_image_VCONV->Val}. As $P$ is a diffeomorphism, $P(B)$ is a compact neighborhood of $P(A)$, so using Proposition \ref{proposition_approximation_by_smooth_val_in_Val_with_restrictions_on_vsupp}, we can find a sequence $(\mu_j)_j$ in $\Val_{k,P(B)}(V^*\times\R)\cap \Val(V^*\times\R)^{sm}$ such that $(\mu_j)_j$ converges to $T(\mu)$. Then $(T^{-1}(\mu_j))_j$ is a sequence in $\VConv_{k,B}(V)$ that converges to $\mu$ in $\VConv_{k,B}(V)$ and by Proposition \ref{proposition_characterization_smooth_valuations_VConv}, $T^{-1}(\mu_j)\in \VConv(V)^{sm}$. The claim follows.
		\end{proof}

	\subsection{Application to invariant valuations}	
		\label{section:Application_to_invariant_valuations}	
		\begin{proposition}
			\label{proposition:density_invariant_valuations}
			Let $G\subset \GL(V)$ be a compact subgroup. Then the space of smooth $G$-invariant valuations is sequentially dense in the space of all $G$-invariant valuations in $\VConv_k(V)$.
		\end{proposition}
		\begin{proof}
			Without loss of generality we can assume that $G$ is a subgroup of $\mathrm{O}(n)$, $V=\R^n$. Let $\mu\in\VConv_k(\R^n)$ be a $G$-invariant valuation and let $R>0$ be such that $B_R$ is a neighborhood of $\supp\mu$. From Proposition \ref{proposition_support_convex_valuation}, it is easy to deduce that $G$ maps an element of $\VConv_{B_R}(\R^n)$ to an element of the same space. Using Theorem \ref{theorem_density_smooth_valuations}, choose a sequence $(\mu_j)_j$ of smooth valuations converging to $\mu$ such that the supports of the valuations $\mu_j$ are all contained in $B_R$. By Theorem \ref{theorem:Characterization_smooth_dually_epi_valuations}, each $\mu$ can be represented by a vertically translation invariant differential form $\tau_j$ with horizontally compact support. By averaging $\mu_j$ with respect to the Haar measure, we obtain a $G$-invariant valuation $\tilde{\mu}_j\in\VConv_{B_R}(\R^n)$. We claim that this valuation is induced by the differential form $\tilde{\tau}_j$ obtained by averaging $g\mapsto \sign(\det g)(g^{-1})^*\tau_j$ with respect to the Haar measure. Using the relation $D(f\circ g)[\tau]=\sign(\det g) D(f)[(g^{-1})^*\tau]$ from Proposition \ref{proposition_Fu_differential_cycle_and_diffeomorphisms}, this is easily verified. Thus $\tilde{\mu}_j$ is a smooth $G$-invariant valuation.\\
			It is easy to see that $G$ acts continuously on $\VConv_{B_R}(\R^n)$, i.e. the map 
			\begin{align*}
				G\times \VConv_{B_R}(\R^n)&\rightarrow \VConv_{B_R}(\R^n)\\
				(g,\mu)&\mapsto [f\mapsto (g\cdot \mu)(f):=\mu(f\circ g)]
			\end{align*}
			is continuous. As $\VConv_{B_R}(\R^n)$ is a Banach space due to Proposition \ref{proposition_Banach_structures_subspaces_VConv}, the principle of uniform boundedness implies that there exists $C>0$ such that
			\begin{align*}
				\|g\cdot\mu\|\le C\|\mu\|\quad\forall g\in G, \mu\in\VConv_{B_R}(\R^n)
			\end{align*}
			(in fact, $G$ acts by isometries, see the definition of the norms in \cite{Knoerr:support_of_dually_epi-translation_invariant_valuations}), were $\|\cdot\|$ denotes the corresponding norm on $\VConv_{B_R}(\R^n)$.  We thus obtain
			\begin{align*}
				\left\|\mu-\tilde{\mu}_j\right\|=&\left\|\int_G g\cdot(\mu -\mu_j)dg\right\|\le \int_G\left\|g\cdot(\mu-\mu_j)\right\|dg\le\int_GC\left\|\mu-\mu_j\right\|dg
				=C\left\|\mu-\mu_j\right\|.
			\end{align*}
			Thus $(\tilde{\mu}_j)_j$ is a sequence of smooth $G$-invariant valuations converging to $\mu$.
		\end{proof}
		Note that the argument actually shows the following
		\begin{corollary}
			If $G\subset \GL(V)$ is a compact subgroup, then any $G$-invariant valuation $\mu\in\VConv_k(V)^{sm}$ can be represented by a $G$-invariant differential form $\tau\in \Omega^{n-k,k}$.
		\end{corollary}
	
		\begin{proof}[Proof of Theorem \ref{maintheorem_classification_transitive_group_1-hom_case}]
			Obviously any $\mathrm{O}(n)$-invariant valuation is $G$-invariant. For the converse inclusion, observe that any $G$-invariant, smooth,  $1$-homogeneous valuation is $\mathrm{O}(n)$-invariant by Corollary \ref{corollary:1-homogeneous_invariant-valuation}. As the operation of $\mathrm{O}(n)$ on $\VConv_1(\R^n)$ is continuous and smooth $G$-invariant valuations are dense in the space of all continuous $G$-invariant valuations by Proposition \ref{proposition:density_invariant_valuations}, the general case thus follows by approximation.
		\end{proof}

\bibliography{literature}

\Addresses
\end{document}